\documentclass[12pt]{amsart}
\usepackage{amsmath, amsfonts,amsthm,times,graphics}
\input amssym.def
\input amssym

\begin{document}

\newtheorem{theorem}{Theorem}[section]
\newtheorem{lemma}[theorem]{Lemma}
\newtheorem{corollary}[theorem]{Corollary}
\newtheorem{proposition}[theorem]{Proposition}
\newtheorem{conjecture}[theorem]{Conjecture}
\newtheorem{question}[theorem]{Question}
    \theoremstyle{definition}
\newtheorem{definition}[theorem]{Definition}
\newtheorem{example}[theorem]{Example}
\newtheorem{xca}[theorem]{Exercise}
\newtheorem{remark}[theorem]{{\it Remark}}

\newcommand{\abs}[1]{\lvert#1\rvert}

 \makeatletter

\title[Curious conjectures on the distribution of 
primes\ldots]
{Curious conjectures on the distribution of 
primes\\ among the sums of the first $2n$  primes}

\author{Romeo Me\v strovi\' c}
\address{Maritime Faculty Kotor, University of Montenegro, Dobrota 36,
85330 Kotor, Montenegro} \email{romeo@ac.me}

\makeatother

{\renewcommand{\thefootnote}{}
\setcounter{footnote}{0}}

 \maketitle
 \begin{abstract} Let $p_n$ be $n$th prime, and 
let $(S_n)_{n=1}^\infty:=(S_n)$ be the sequence of the sums of 
the first $2n$ consecutive primes, that is, $S_n=\sum_{k=1}^{2n}p_k$
with $n=1,2,\ldots$. Heuristic arguments supported by   
the corresponding computational results suggest that the primes are distributed 
among sequence $(S_n)$ in the same way that they are distributed 
among positive integers. In other words,  taking into account the Prime
Number Theorem, this assertion is equivalent to
  \begin{equation*}\begin{split}
&\# \{p:\, p\,\,{\rm is\,\,a\,\, prime\,\,
and}\,\, p=S_k \,\,{\rm for\,\,some\,\,} k 
\,\,{\rm with\,\,} 1\le k\le n\}\\
\sim & \# \{p:\, p\,\,{\rm is\,\,a\,\, prime\,\,
and}\,\, p=k \,\,{\rm for\,\,some\,\,} k 
\,\,{\rm with\,\,} 1\le k\le n\}\sim\frac{\log n}{n}\,\, {\rm as}\,\,
n\to\infty, 
  \end{split}\end{equation*}
where $|S|$ denotes the cardinality of a set $S$. Under the assumption that
this assertion is true (Conjecture 3.3), we say that  $(S_n)$ satisfies the 
 Restricted Prime Number Theorem.
Motivated by this, in  Sections 1 and 2  we give some definitions, 
results and examples  concerning the generalization 
of the prime counting function $\pi(x)$ to    
increasing positive integer sequences.

The remainder  of the paper (Sections 3-7) is devoted    
to the study of  mentioned sequence $(S_n)$. Namely, 
we propose several conjectures and we prove  their consequences 
concerning the  distribution of primes in
the sequence $(S_n)$. These conjectures are mainly  motivated by the Prime 
Number Theorem,
some heuristic arguments and related  computational results.
Several consequences of these conjectures are also established.   
 \end{abstract}

  \section{Introduction, Motivation and Preliminaries}
An extremely  difficult problem in number theory is the 
{\it distribution of the primes} among the natural numbers. This problem 
involves the study of the asymptotic behavior of the counting function
$\pi(x)$ which is one of the more intriguing functions in number theory. 
The function $\pi(x)$ is defined as the number of primes $\le x$.
For elementary methods in the study of the distribution 
of prime numbers, see \cite{di4}.

Although questions in number theory were not always mathematically 
{\it en vogue}, by the middle of the nineteenth century the problem of 
counting primes had attracted the attention of well-respected mathematicians
such as Legendre, Tch\'{e}bychev, and the prodigious Gauss.

A query about the frequency with which primes occur elicited the following 
response:

{\it I pondered this problem as a boy, and determined that, at around $x$,
the primes occur with density $1/\log x$}--C. F. Gauss 
(letter to Encke, 24 December 1849).
 Gauss wrote:

This remark of Gauss can be interpreted as predicting that
   $$
\# \{{\rm primes}\,\le x\} \approx \sum_{n=2}^{\lfloor x \rfloor}
\frac{1}{\log n}\approx
\int_{2}^x\frac{dt}{\log t}={\rm Li}(x).  
   $$
Studying tables of primes,
 C. F. Gauss in the late 1700s and A.-M. Legendre in the early 1800s 
conjectured the celebrated {\it Prime Number Theorem}:
 $$
\pi(x)=|\{p\le x:\, p\,\,{\rm prime}\}|\sim \frac{x}{\log x}
  $$
(here, as always in the sequel,  $|S|$ denotes the cardinality of a set $S$). 
  
This theorem was proved 
much later (\cite[p. 10, Theorem 1.1.4]{cp}; for its simple analytic 
proof see \cite{new} and \cite{za}, and for its history see \cite{bd1},
 \cite{golds}, \cite{gr} and \cite[p. 21]{me}.
Briefly, $\pi (x)\sim x/\log x$ as $x\to\infty$, or in other words,
 the density of primes $p\le x$ is $1/\log x$; that is, the ratio
$\pi(x): \left(x/\log x\right)$ converges to 1 as $x$ grows without bound.
Using L'H\^{o}pital's rule, Gauss showed that the logarithmic integral 
$\int_{2}^x dt/\log t$, denoted by $\mathrm{Li}(x)$, is asymptotically
equivalent to $x/\log x$. Recall that Gauss felt that 
$\mathrm{Li}(x)$ gave better approximations to $\pi(x)$ than 
$x/\log x$ for large values of $x$.
 
 Though unable to prove the Prime Number Theorem, 
several significant contributions to the proof of Prime Number Theorem 
were given by P. L. Chebyshev in his two important 1851--1852 papers 
(\cite{che1} and \cite{che2}). Chebyshev proved that there exist positive 
constants $c_1$ and $c_2$ and a real number $x_0$ such that 
  $c_1x/\log x\le\pi(x)\le c_1x/\log x$ for $x>x_0$.  In other words, 
$\pi(x)$ increases as $x\log x$. Using methods of complex analysis and 
the ingenious ideas of Riemann (forty years prior),
 this  theorem was first proved in 1896, independently 
 by J. Hadamard and C. de la Vall\'{e}e-Poussin  
(see e.g., \cite[Section 4.1]{pol}).

A {\it generalized prime system} (or $g$-{\it prime system}) 
${\mathcal G}$ is a sequence of positive 
real numbers $q_1,q_2,q_3,\ldots$ satisfying 
   $1<q_1\le q_2\le \cdots \le q_n\le q_{n+1}\le \cdots$
and  $q_n\to \infty$ as $n\to\infty$. From these can be formed the 
system ${\mathcal N}$ of {\it generalized integers} or {\it Beurling integers};
that is, the numbers of the form $q_1^{k_1}q_2^{k_l}\cdots q_l^{k_l}$,
where $l\in\Bbb N$ and $k_1,k_2,\ldots,k_l\in\Bbb N_0:=
\Bbb N\bigcup \{0\}$. Notice that ${\mathcal N}$ denotes the 
{\it multiplicative  semigroup} generated  by ${\mathcal G}$,
and it consists of the unit 1 together with all finite power-products of
$g$-primes, arranged in increasing order and counted with multiplicity. 

Clearly, this system generalizes the notion of 
primes and positive integers obtained from them. Such systems were 
first  introduced by A. Beurling \cite{be} and have been studied by many 
authors since then  (see  in particular \cite{bd2}, \cite{ba},
\cite{di3}, \cite{dmv}, \cite{ma}, \cite{ny} and \cite{zh1}). 
In particular, Nyman \cite{ny} and Malliavin \cite{ma} 
sharpened Beurling's results in various ways. 

Much of the theory concerns connecting the asymptotic behaviour 
of  {\it g-prime counting function} and  
{\it g-
counting function} 
 $\pi_{\mathcal G}(x)$ and $N_{\mathcal G}(x)$, defined on $[1,\infty)$ 
respectively by
   $$
\pi_{\mathcal G}(x)=\sum_{q\in \mathcal G, q\le x} 1\quad{\rm and}\quad
N_{\mathcal G}(x)=\sum_{n\in \mathcal N, n\le x} 1,
   $$
where in the first sum the summation is taken over all $g$-primes,
 counting multiplicities. Similarly, for the second sum 
$\sum_{n\in \mathcal N, n\le x} 1$. Accordingly, we have
$$
\pi_{\mathcal G}(x)=\# \{i:\, q_i\in \mathcal G\,\,{\rm  and}\,\, g_i\le x\}
\quad{\rm and}\quad
N_{\mathcal G}(x)=\# \{i:\, n_i\in \mathcal N\,\,{\rm  and}\,\, n_i\le x\}.
   $$
If $\mathcal G=\{a_1,a_2,\ldots ,a_n,a_{n+1},\ldots\}=(a_n)_{n=1}^{\infty}$
is a sequence such that $a_1\le a_2\le \cdots \le a_n\le a_{n+1}\le \cdots$, then
obviously, we have $\pi_{\mathcal G}(a_n)=n$ for each $n\in\Bbb N$. 
 
In 1937 Beurling proved  \cite[Th\'{e}or\`{e}me IV]{be} 
 that if $N_{\mathcal G}$ satisfies the asymptotic relation 
$N_{\mathcal G}(x)=Ax+O(x/\log^{\gamma}x)$ 
with some constants $A>0$ and $\gamma >3/2$, then the number 
of $q_n$'s such that $q_n\le x$ is equal to $x/\log x +o(x/\log x)$, i.e., 
   $$
\pi_{\mathcal G}(x)=\frac{x}{\log x}+o\left(\frac{x}{\log x}\right).
   $$  
In other words,   the conclusion of the 
Prime Number Theorem (in the sequel, shortly written as PNT)  is valid for such a system $\mathcal G$
(from this reason  often called a {\it Beurling prime number system}).  
Beurling also gave an example in which 
$N_{\mathcal G}(x)=Ax+O(x/\log^{3/2}x)$ but PNT is not valid.
This result was refined in 1969 by H. G. Diamond \cite[Theorem (B)]{di1}.
In 1970 Diamond \cite{di2}  also proved that Beurling's condition 
is sharp, namely, the PNT does not necessarily hold if $\gamma=3/2$.

In particular, if ${\mathcal G}$ is 
a  set ${\mathcal P}:=\{p_1,p_2,p_3,\ldots\}$ of all primes 
$2=p_1<p_2<p_3<\cdots$ with the associated  
multiplicative  semigroup $\mathcal N=\Bbb N=\{1,2,3,\ldots\}$,   
then PNT states that 
       $$
\pi(x)\sim \frac{x}{\log x},\quad {\rm as}\,\, x\to\infty,
    $$
where $\pi(x)$ is the  usual {\it prime counting function}, that is,
   $$
\pi(x)=\sum_{p\,\, {\rm prime},\,\, p<x} 1.
 $$
As observed in \cite[Introduction]{bd2}, the additive structure 
of the positive integers is not particularly relevant to the distribution 
of primes. 
Therefore, for a given $g$-prime system ${\mathcal G}$
defined above, it can be of interest to consider   
the distribution of $g$-primes (the elements in ${\mathcal G}$)  
with respect to certain associated system of generalized integers 
without any algebraic (multiplicative) structure.
This means that the associated system ${\mathcal N}$ to ${\mathcal G}$
defined above may be some subset of $[1,+\infty)$ which is 
not a multiplicative semigroup (generated by ${\mathcal G}$). 

In particular, here we mainly  consider the case when 
${\mathcal G}$ is an infinite set of primes and 
the associated system ${\mathcal N}$ to ${\mathcal G}$
is an increasing integer sequence $(a_n)_{n=1}^{\infty}$. 
We focus  our attention when  ${\mathcal G}$ is a set of all primes 
whose associated system ${\mathcal N}$ is the sequence 
$(a_n)_{n=1}^{\infty}:=(\sum_{i=1}^{2n}p_{i})_{n=1}^{\infty}$ 
where $2=p_1<p_2<\cdots <p_n<\cdots$ are all the primes.

Let  $({\mathcal G}, {\mathcal N}:=(a_k)_{k=1}^{\infty})$
be a pair defined above. Then we  define its {\it 
counting function} 
$N_{{\mathcal G},(a_k)}(x)$ as 
  $$
N_{{\mathcal G},(a_k)}(x)=\#\{i:\,i\in\Bbb N\,\,
{\rm and}\,\, a_i\le x\}.
   $$
Furthermore, the  {\it prime counting function}
for $({\mathcal G}, {\mathcal N}:=(a_k)_{k=1}^{\infty})$
is the function $x\mapsto \pi_{{\mathcal G},(a_k)}(x)$ defined on $[1,\infty)$  as
  \begin{equation}
\pi_{{\mathcal G},(a_k)}(x)=\# \{q:\, q\in {\mathcal G}\,\,
{\rm and}\,\, q=a_i \,\,{\rm for\,\,some\,\,} i 
\,\,{\rm with\,\,} a_i\le x\}.
  \end{equation}
Heuristic and computational results 
show that for many ``natural pairs''  
$({\mathcal G}, {\mathcal N}:=(a_k)_{k=1}^{\infty})$
the associated 
counting function  $N_{{\mathcal G},(a_k)}(x)$
has certain asymptotic growth as $x\to\infty$. 
Notice that for each $n\in \Bbb N$ we have
   \begin{equation}\label{(2)}
\pi_{{\mathcal G},(a_k)}(a_n)=\# \{q:\, q\in {\mathcal G}\,\,
{\rm and}\,\, q=a_i \,\,{\rm for\,\,some\,\,} i 
\,\,{\rm with\,\,} 1\le i\le n\}.
  \end{equation}
The {\it normalizable prime counting function} for 
$({\mathcal G}, {\mathcal N}=(a_k)_{k=1}^{\infty})$
is the function $(n,x)\mapsto p_{{\mathcal G},(a_k)}(n,x)$ defined for 
$(n,x)\in \Bbb N\times [1,+\infty)$ as
 \begin{equation}\label{(3)}
p_{{\mathcal G},(a_k)}(n,x)=\frac{\log a_n}{a_n}\pi_{{\mathcal G},(a_k)}(x).
 \end{equation}
The above expression induces the  companion  sequence 
$(b_n)_{n=1}^{\infty}$ of $(a_k)_{k=1}^{\infty}$ 
defined as 
 \begin{equation}\label{(4)}
b_n=p_{{\mathcal G},(a_k)}(a_n)=\frac{\log a_n}{a_n}\pi_{{\mathcal G},(a_k)}(a_n),
\,\, n=1,2,\ldots.
 \end{equation}
We also define another ``companion'' sequence $(c_n)_{n=1}^{\infty}$ of 
$(a_k)_{k=1}^{\infty}$ defined as
 \begin{equation}\label{(5)}
c_n=\frac{a_n}{(\log n)(\log a_n)}=
\frac{\pi_{{\mathcal G},(a_k)}(a_n)}{b_n\log n},\,\, 
n=1,2,\ldots.
 \end{equation}

Here as always in the sequel, we will suppose that 
${\mathcal G}$ is a set of all primes whose associated system 
${\mathcal N}$ is a sequence $(a_k)_{k=1}^{\infty}$. 
For brevity, in the sequel the functions defined by (1), (2), (3) and (4) 
will be denoted  by  $p_{(a_k)}(x)$,  $\pi_{(a_k)}(a_n)$,  $p_{(a_k)}(n,x)$ and 
$p_{(a_k)}(a_n)$, respectively.

 \begin{definition}
 Let $\Omega$ be a set of all   nonnegative continuous real functions
defined on $(1,+\infty)$ and let $(a_k)_{k=1}^{\infty}:=(a_k)$ be an increasing
sequence of positive integers. We say that $(a_k)$ is a 
{\it prime-like sequence}
if there exists the function $\omega_{(a_k)}=\omega\in\Omega$ such that 
the function $n\mapsto \pi_{(a_k)}(a_n)$ defined by (2) is asymptotically 
equivalent to $\omega(n)$ as $n\to \infty$. 
Then we say that a sequence $(a_k)$ satisfies the
  $\omega$-{\it Restricted Prime Number Theorem}.
   
In particular, if $\omega(x)\sim x/\log x$ as $x\to\infty$,
then we say that a sequence $(a_k)$ satisfies the 
{\it Restricted Prime Number Theorem} (RPNT). 
  \end{definition}

  \begin{proposition} 
Let $(a_k)_{k=1}^{\infty}$ be a positive integer sequence, and let 
$(b_n)_{n=1}^{\infty}$ be its companion sequence defined by $(4)$. Then
   \begin{equation}\label{(6)}
\limsup_{n\to\infty}b_n\le 1.
  \end{equation}
   \end{proposition}
\begin{proof} 
Taking  the obvious inequality $\pi_{(a_k)}(a_n)\le\pi(a_n)$
with $n=1,2,\ldots$ into (4) we get 
  \begin{equation*}
b_n\le \frac{\pi(a_n)\log a_n}{a_n},\,\, n=1,2,\ldots,
 \end{equation*}
which by the Prime Number Theorem immediately yields
  \begin{equation*}
\limsup_{n\to\infty}b_n\le \lim_{n\to\infty}
\frac{\pi(a_n)\log a_n}{a_n}=1,
 \end{equation*} 
as desired.
   \end{proof}

\begin{proposition} Let $(a_n)$ be a prime-like sequence with
the associated function $\omega(x)$. Then
  \begin{equation}\label{(7)}
\limsup_{n\to\infty}\frac{\omega(n)}{\pi(a_n)}\le 1.
 \end{equation}
This means that $\omega(n)$ grows slowly than $\pi(a_n)$ as $n\to\infty$.
\end{proposition}
 \begin{proof}
Notice that the inequality (7) is equivalent to 
  \begin{equation}\label{(8)}
\limsup_{n\to\infty}\frac{\log a_n}{a_n}\pi_{(a_k)}(a_n)\le 1.
 \end{equation}
Since by the assumption, $\omega(n)\sim \pi_{(a_k)}(a_n)$,
the inequality (8) yields 
   \begin{equation*}
\limsup_{n\to\infty}\frac{\log a_n}{a_n}\omega(n)\le 1,
 \end{equation*}
whence, in view of the fact that $\log a_n/a_n\sim 1/\pi(a_n)$,
immediately follows (7).
 \end{proof}
  \begin{remark} 
 The inequality $(6)$ is sharp since by the Prime Number Theorem (see 
Example 2.1), 
equality in $(6)$ holds for the sequences $a_k=k$ 
with $k=1,2,\ldots$.
  \end{remark}

  \begin{remark}  
If a sequence $(a_k)$ satisfies the 
$\omega$-{\it Restricted Prime Number Theorem},
then by (4) we have
 \begin{equation}\label{(9)}
\omega(n)\sim \pi_{(a_k)}(a_n)=\frac{a_nb_n}{\log a_n}\quad 
{\rm as}\,\, n\to\infty.
  \end{equation}
  \end{remark}

\begin{remark}  Let  $(a_k)$ be a sequence  satisfying  
the  $\omega$- Restricted Prime Number Theorem. Then clearly,
 $\omega_{(a_k)}(n)\le n$  for all sufficiently large  $n$.
Moreover, $\omega_{(a_k)}(n)\sim n$ as $n\to \infty$ if and only if 
the density  of primes in a sequence $(a_k)$  is equal to 1.

Notice also that by the Prime Number Theorem, 
$\omega_{\Bbb N}(x)=x/\log x$ for the sequence of all positive 
integers $\Bbb N=\{1,2,\ldots n,\ldots\}$, that is, 
$\Bbb N$ satisfies the Restricted Prime Number Theorem
(cf. Conjecture 3.3). 
\end{remark}

Here, as always in the sequel,
 ${\mathcal P}=(p_n):=\{p_1,p_2,\ldots p_n,\ldots\}$ will denote the set of
all primes, where $2=p_1<3=p_2<p_3<\cdots <p_n<\cdots$.
Moreover, $(a_n)$ will always denotes an infinite strictly increasing sequence 
of positive integers. Hence, for such a sequence must be
$a_n\ge n$ for each $n\in\Bbb N$.

The remainder of the paper is organized as follows. In Section 2 we present 
five examples  
concerning the determination of the function $\omega_{(a_k)}(x)$
and a sequence $(b_n)$ associated to a given sequence $(a_k)$.
In particular, we consider the sequence $(a_k)_{k=1}^\infty$ with 
$a_k=a+(k-1)d$, where $a\ge 1$ and $d>1$ are 
relatively prime integers. 

In Section 3 we consider the distribution of primes in the sequence 
$(S_n)_{n=1}^{\infty}$ whose terms are given by $S_n=\sum_{i=1}^{2n}p_i$,
where $p_i$ is the $i$th prime. 
Heuristic arguments supported by related computational results 
suggest the curious conjecture that the sequence $(S_n)$ satisfies 
the Restricted Prime Number Theorem (Conjecture 3.3). In other words, 
this means that the  primes are distributed 
amongst all the terms of the sequence $(S_n)$ in the same way that they are 
distributed amongst all the positive integers. Under this conjecture,
we prove that if $q_k$ is the $k$th prime in $(S_n)_{n=1}^\infty$, then 
$q_k\sim 2k^2\log^3 k\sim 2p_k^2\log k$ as $k\to\infty$ 
(Corollaries 3.6 and 3.7).   

Assuming that Conjecture 3.3 is true, in Section 4 we give the asymptotic 
expression for the $k$th prime in the sequence $(S_n)$ (Corollary 4.2);
namely, $q_k\sim 2k^2\log^3 k$ as $k \to\infty$. This result is refined by 
Theorem 4.4.  
 We also conjecture that $\lfloor k\log k \rfloor +1\le m$ for each pair 
$(k,m)$ of positive integers  with $k\ge 1$ and $q_k=S_m$ (Conjecture 
4.6). Some consequences of Conjectures 3.3 and 4.6  are also presented.

Section 5 is devoted to the estimations of the values 
$M_k$  $(k=1,2,\ldots)$ involving 
in the expression for $q_k$ from Theorem 4.4. We also propose 
some other  conjectures concerning the sequences $(S_k)$ and
$(M_k)$. Related consequences are also established.

The conjectures presented in this paper, as well as some their consequences,
are mainly supported  by some computational results given in Section 
6. In particular, the number $\pi_n:=k$ of primes in the set 
${\mathcal S}_n:=\{S_1,S_2,\ldots ,S_n\}$ for 38 values 
of $n$ up to $10^9+5\cdot 10^8$ are presented  
in Table 1. For such values $k$ and the associated 
indices  $m$ such that $q_k=S_m$, the corresponding approximate values 
of $q_k$,  $M_k$ (together with lower and upper bounds of $M_k$),
$(k\log k)/m$ and $S_m\sqrt{k\log k}/(2m^{5/2}\log m)$ are  
also given in this table. Under the previous  notations, 
related numerical results 
for $q_k/(2k^2\log^3 k)$, $q_k/(2m^2\log m)$ and two 
estimates involving $q_k$ which are discussed in Section 4, 
are given in Table 3. Some additional computational results,  
the conjectures and their consequences 
are also given in Section 6.       

In the last Section 7 we propose the stroner (asymptotic) 
version of Conjecture 3.3 which coincides with well known form 
of Prime Number Theorem involving the function ${\rm li} (x)$.   

Notice that similar considerations to those for the sequence  $(S_n)$
 concerning  alternating sums of consecutive primes are 
given in \cite{me2}. 

\section{Examples}

\begin{example} 
For the sequence $(a_k)$ with $a_k=k$ 
$(k=1,2,\ldots)$, we clearly have $\pi_{(k)}(n)=\pi(n)$,
and hence
\begin{equation}\label{(10)}
b_n=\frac{\pi(n)\log n}{n},\,\, n=1,2,\ldots.
 \end{equation}
By the Prime Number Theorem, from $(10)$ we find that 
\begin{equation}\label{(11)}
\lim_{n\to\infty} b_n=\lim_{n\to\infty}\frac{\pi(n)\log n}{n}=1.
 \end{equation}
\end{example}

\begin{example}  Let  $(a_k)$ be a sequence of all primes, that is, 
$a_k=p_k$ with $k\in\Bbb N:=\{1,2,\ldots\}$, 
where $p_k$ is the $k$th prime. Since by $(1)$, 
$\pi_{(p_k)}(p_n)=\pi(p_n)=n$, substituting this into $(4)$ yields
 \begin{equation}\label{(12)}
b_n=\frac{n\log p_n}{p_n},\,\, n=1,2,\ldots.
 \end{equation}
Now applying to \eqref{(12)}  the  well known  
fact that $p_n\sim n\log n$ as $n\to\infty$ 
(see, e.g., 
\cite{ms}),  we find that
  \begin{equation}\label{(13)}
\lim_{n\to\infty} b_n=1.
 \end{equation}
Notice also that the  known inequality $p_n>n\log n$ with 
$n\ge 1$ (see, e.g., \cite[(3.10) in Theorem 3]{rs1}) 
implies that $b_n< 1$ for all $n\ge 1$. 
    \end{example}

\begin{example} Suppose that $a$ and $d$ are 
relatively prime positive integers.
Then concerning Dirichlet's theorem de la Vall\'{e}e Poussin established
(see, e.g., \cite[p. 205]{r1}) 
that the number of primes $p<x$ with $p\equiv a(\bmod{\, d})$
is approximately 
  \begin{equation}\label{(14)}
\frac{\pi(x)}{\varphi(d)}\sim \frac{1}{\varphi(d)}\cdot 
\frac{x}{\log x}.
   \end{equation}
Here $\varphi(n)$ is the Euler totient function defined as the 
number of positive integers not exceeding $n$ and relatively prime to $n$.
Note that the right hand side of (14) is the same for any $a$ such that 
$\gcd(a,d)=1$. This shows that primes are in a certain sense 
uniformly distributed in reduced residue classes with respect to a fixed 
modulus. Notice that for a sequence $(a_k)_{k=1}^{\infty}$ 
given by $a_k=a+(k-1)d$, (14) can be written as   
    \begin{equation}\label{(15)}
\pi_{(a_k)}(a_n)\sim \frac{\pi(a_n)}{\varphi(d)}\,\,{\rm as}\,\, n\to \infty. 
    \end{equation}
Inserting (15) together with $\pi(a_k)\sim a_k/\log a_k$
into \eqref{(4)} immediately gives 
  \begin{equation}\label{(16)}
\lim_{n\to\infty} b_n=\frac{1}{\varphi(d)}\lim_{n\to\infty}\frac{\pi(a_n)
\log a_n}{a_n}=\frac{1}{\varphi(d)}.
 \end{equation}
Then  substituting (16) into (9), we obtain that for the associated 
 function $\omega_{a,d}(x):=\omega_{(a_k}$ of the sequence $(a_k)$ there holds 
    \begin{equation}\label{(17)}
\omega_{a,d}(n)\sim \pi_{(a_k)}(a_n)
\frac{a_nb_n}{\varphi(d)\log a_n}=\frac{a+(n-1)d}{\varphi(d)\log (a+(n-1)d)}
\sim \frac{dn}{\varphi(d)\log n}\,\,{\rm as}\,\, n\to \infty.
       \end{equation} 
It follows that $\omega_{a,d}(x)= dx/(\varphi(d)\log x)$ for 
$x\in(1,+\infty)$.
  \end{example}

  \begin{example}  Let $(a_n)$ be a sequence defined as 
$a_n=2^{p_n}-1$, where $p_n$ is the $n$th prime. The numbers $a_n$
are called {\it Mersene numbers}. A prime 
that appears  in the sequence $(a_n)$ is called {\it Mersenne prime}.  
Namely, it is easy to show (see, e.g., \cite[p. 28]{r2}) that if $2^n-1$ 
is prime, then so is $n$. 
The greatest known Mersenne prime is $2^{43112609}-1$ with the exponent 
43112609 (12978169 digit number), and it is discovered in August 2008. 
This is in fact one between 45 known Mersenne primes, and so 
$a_{45}\le 2^{43112609}-1$.

In 1980 H. Lenstra and C. Pomerance, working independently,
came the conclusion that the probability that a Mersenne number 
$2^p-1$ is prime is $e^{\gamma}\log (ap)/(p\log 2)$
with $\gamma= 0.577216\ldots$ (the Euler-Mascheroni constant), where $a=2$ if 
$p\equiv 3(\bmod{\,4})$ and  $a=6$ if  $p\equiv 1(\bmod{\,4})$.
Recall that the constant $e^\gamma=1.781072\ldots$ is important 
in number theory; namely, 
$e^\gamma=\lim_{n\to\infty}\frac{1}{\log p_n}\prod_{k=1}^n\frac{p_k}{p_k-1}$
which restates the third of Mertens' theorems (\cite{mer}, also see 
\cite[pp. 351--353, Theorem 428]{hw}). 
Then notice that the distribution of the $\log $ of the Mersenne 
primes is a {\it Poisson Process} (see \cite{wa}).

Accordingly to the above assumption given by Lenstra and  Pomerance, 
if $a_k=2^{q_k}-1$, where $(q_k)_{k=1}^\infty$ is a sequence 
of all primes $\equiv 3(\bmod{\,4})$  ($q_1=3,q_2=7,q_3=11,\ldots$), 
 for the associated  function $\omega^{(3,4)}(x)$ to $(a_k)$
we have that ``the expected number'' of primes between the 
first $n$ terms of the sequence $(q_k)$ is
   \begin{equation}\label{(18)}
\sim\omega^{(3,4)}(n)\sim \sum_{k=1}^n\frac{e^{\gamma}\log (2q_k)}{q_k\log 2}
\,\,{\rm as}\,\, n\to \infty.
     \end{equation} 
Since  $q_k\sim p_{2k}\sim 2k\log k$, 
substituting this into (18) and using the well known asymptotic formula
$\sum_{k=1}^n 1/k\sim \gamma +\log n$ as $n\to\infty$,  we get

  \begin{equation}\label{(19)}
  \begin{split}
\omega^{(3,4)}(n) & \sim 
\frac{e^{\gamma}}{2\log 2}\sum_{k=2}^n\frac{\log (4k\log k)}{k\log k}=
\frac{e^{\gamma}}{2\log 2}\sum_{k=2}^n\frac{\log k+\log 4+\log\log k}{k\log k}\\
&\sim \frac{e^{\gamma}}{2\log 2}\left(
\sum_{k=2}^n\frac{1}{k}+\sum_{k=2}^n\frac{\log 4}{k\log k}
+\sum_{k=2}^n\frac{\log\log k}{k\log k}\right)\\
&\sim \frac{e^{\gamma}}{2\log 2}\left((\gamma +\log n)+
\log 4\int_{2}^n\frac{dx}{x\log x}+
\int_{2}^n\frac{\log\log x}{x\log x}\,dx\right)\\
&\quad (\text{the\,\, changes}\,\, \log x=s\,\, \text{and}\,\, \log\log x=t\\
&= \frac{e^{\gamma}}{2\log 2}\left(\log n+
\int_{\log 2}^{\log n}\frac{ds}{s}+\int_{\log\log 2}^{\log\log n}t\,dt\right)\\
&\sim \frac{e^{\gamma}}{2\log 2}\left(\log n+\log\log n+
\frac{(\log\log n)^2}{2}\right)\,\,{\rm as}\,\, n\to \infty.
 \end{split}
  \end{equation}

This shows that $\omega^{(3,4)}(x)=
e^{\gamma}\left(\log x+\log\log x
+(\log\log x)^2/2\right)/(2\log 2)$, and hence $\pi_{(a_k)}(a_n)\sim 
e^{\gamma}/(2\log 2)\left(\log n+\log\log n+(\log\log n)^2/2\right)$. 
Substituting this in (4), 
where $(b_n)$ is the companion' sequence of $(a_k)$, 
and using the fact that $q_n\sim  2n\log n$, we find that
   \begin{equation}\label{(20)}  
b_n\sim \frac{e^{\gamma}n(\log n)^2}{4^{n\log n}}\,\,{\rm as}\,\, n\to \infty.
   \end{equation}

Similarly, under the  above assumptions attributed by Lenstra and Pomerance, 
if $a_k'=2^{r_k}-1$, where $(r_k)_{k=1}^\infty$ is a sequence 
of all primes $\equiv 1(\bmod{\,4})$  ($r_1=5,r_2=13,r_3=17,\ldots$), 
then for the associated  function $\omega^{(1,4)}(x)$ to $(a_k')$
and the companion sequence $(b_n')$ of $(a_k')$ the same relations 
(18)--(20) are satisfied.
\end{example}

\begin{example}  Let  $(a_k)_{k=1}^\infty$ be an increasing  sequence of 
positive integers satisfying 
 \begin{equation}\label{(21)}
\frac{\log a_k}{a_k}=o(k^{-1}).
  \end{equation}
Then from (4) and the obvious fact that $\pi_{(a_k)}(a_n)\le n$ for each 
$n\in\Bbb N$, we find that 
 \begin{equation}\label{(22)}
\lim_{n\to\infty}b_n=0.
 \end{equation}
In particular, $(22)$ holds for any sequence $(a_k)$ 
satisfying one of the following asymptotics:
$a_n\sim a^n$ with a fixed $a>1$; $a_n\sim n\log^\alpha n$ with $\alpha >1$;
$a_n\sim n^\alpha$ with $\alpha >1$; or $a_n\sim n^\alpha \log^\beta n$ 
with $\alpha \ge 1$ and $\beta>1$.
  \end{example}

Accordingly, we ask the following question.

\begin{question} For what real numbers $\alpha\in (0,1)$
there exists a sequence $(a_k)$ whose companion sequence $(b_n)$
defined by $(4)$  satisfies the limit relation
  $$
\limsup_{n\to\infty}b_n=\alpha ? 
  $$
\end{question}

\section{Distribution of primes in the sequence $(S_n)$ with 
$S_n=\sum_{i=1}^{2n}p_i$}
Here, as always in the sequel, we consider the distribution of primes in the sequence 
$(S_n)_{n=1}^{\infty}$ whose terms are given by $S_n=\sum_{i=1}^{2n}p_i$,
where $p_i$ is the $i$th prime.
Recall that the prime counting function $\pi(x)$ is defined as the number of 
primes $\le x$.

   \begin{proposition}
Let $(S_n)$ be the sequence defined as $S_n=\sum_{i=1}^{2n}p_i$.
Then as $n\to\infty$, 
 \begin{equation}\label{(23)}
S_n\sim 2n^2\log n
\end{equation}
and 
  \begin{equation}\label{(24)} 
\pi(S_n)\sim n^2.
  \end{equation}
Furthermore, if $x$ is a real number such that $S_{n}\le x<S_{n+1}$, then 
 \begin{equation}\label{(25)} 
n\sim \sqrt{\frac{x}{\log x}}\,\, as \,\, n\to\infty.
  \end{equation}
\end{proposition}
  \begin{proof}
 Let $(S_n')$ be the sequence defined as $S_n'=\sum_{i=1}^np_i$
(this is Sloane's sequence A007504 in \cite{sl}).
By the Prime Number Theorem, we have 
(see, e.g., \cite[page 5]{s1}), 
   \begin{equation}\label{(26)}\begin{split}
S_n':=\sum_{i=1}^np_n\sim &\sum_{k=1}^{n} k\log k\sim\int_1^{n} x\log x\,dx=
\frac{x^2}{2}\log x \Big|_1^{n}-\int_{1}^{n}\frac{x^2}{2}(\log x)'\,dx\\
 \sim &\frac{n^2\log n}{2}\,\, {\mathrm as} \,\, n\to\infty.
 \end{split}\end{equation}
It follows from (26) that  
 \begin{equation}\label{(27)}
S_n=S_{2n}'\sim 2n^2\log n,
\end{equation}
which implies (23). By the Prime Number Theorem, from  (27) 
we have 
   \begin{equation}\label{(28)} 
\pi(S_n)\sim\frac{2n^2\log n}{\log (2n^2\log n)}\sim
\frac{2n^2\log n}{\log 2+2\log n+\log\log n}\sim 
n^2.
  \end{equation}
Finally, (25) immediately follows from (23).
  \end{proof}
\begin{remark}  
For refinements of the  estimate (23), see \cite{du1}, \cite{rs2} 
and \cite[Theorem 2.3]{si}).
We see from (23) that there are $\sim n^2$ primes less than $S_n$.
Using this fact,  
Z.-W. Sun \cite[Remark 1.6]{s1}
conjectured that the number of primes in the interval 
$(\sum_{i=1}^np_i,\sum_{i=1}^{n+1}p_i)$ is asymptotically 
equivalent to $n/2$ as $n\to \infty$. Under the validity of this conjecture,
in particular it follows that  the number of primes in the interval 
$(S_n,S_{n+1})$ is asymptotically 
equivalent to $n$ as $n\to \infty$. Moreover, 
we also believe that the ``probability'' that 
$\sum_{i=1}^{2n}p_i$ 
is a prime is $2n/p_{2n}$, which is $\sim 1/\log n$ because of 
$p_{2n}\sim 2n\log 2n$.
Notice that the ``probability" of a large integer  $n$ being a  prime is 
also asymptotically equal to $1/\log n$.

Furthermore, 
some computational results and heuristic arguments show
 that between these $\sim n^2$ 
 primes which  are less than $S_n$ there are $\sim 2n/\log S_n\sim n/\log n$ 
 primes that belong to the set ${\mathcal S}_n:=\{S_1,S_2,\ldots ,S_n\}$.
For example, if $n=10^8$ then $n/\log n=10^8/\log 10^8=5428681.02$,
while from the second column of Table 1 of Section 6 we see that 
there are 5212720 primes in the set ${\mathcal S}_{10^8}$
(cf.  Table 2 of Section 6). 
Accordingly, we propose the following curious conjecture which is basic 
in this paper. 
  \end{remark}

    \begin{conjecture}
The sequence $(S_n)$ with $S_n=\sum_{i=1}^{2n}p_i$ satisfies the Restricted 
Prime Number Theorem. In other words, 
  \begin{equation}\label{(29)}\begin{split}
\pi_n:&=\pi_{(S_k)}(S_n)=\# \{p:\, p\,\,{\rm is\,\,a\,\, prime\,\,
and}\,\, p=S_i \,\,{\rm for\,\,some\,\,} i \,\,{\rm with\,\,} 1\le i\le n\}\\
&\sim \frac{n}{\log n}\quad {\rm as\,\,} n\to\infty.\end{split}
  \end{equation}
  \end{conjecture}
Let us recall that in all  results  of this section
(Corollaries 3.4, 3.6, 3.7, 3.8 and 3.13) we assume the truth of 
Conjecture 3.3.
In particular, Conjecture 3.3 implies  Euclid's theorem 
(on the infinitude of primes) for $(S_n)$ as follows.

    \begin{corollary}[Euclid's theorem for the  sequence $(S_n)$]
The sequence $(S_n)$ contains infinitely many primes.
  \end{corollary}

 \begin{remark}
Notice that the sequence $(S_n)$  is closely 
related to the Sloane's sequence A013918  \cite{sl}  
containing all primes (in increasing order)  equal to the sum of the first $m$ 
primes for some $m\in \Bbb N$ (A013918 is in fact the 
intersection of A000040-the sequence of all primes
and A007504-sum of first $n$ primes). The first few terms of the sequence 
A013918 are: $2,5,17,41,197,281,7699,8893,22039$;
see the related link by T. D. Noe \cite[A013918]{sl} which gives the table of 
the first 10000 terms of this sequence (10000th term is 402638678093).
Notice also that the Sloane's sequence A013916 in \cite{sl} associated to 
the sequence  A013918 gives numbers $n$ such that the sum 
of the first $n$ primes is prime. The first few terms of this sequence
are: $1,2,4,6,12,14,60,64,96$ (see the related link by D. W. Wilson
\cite[A013918]{sl} which gives table of the first 
10000 terms of this sequence (10000th term is 244906).   
Similarly, the second Sloane's  sequence A013917 ($(a_n)$) 
associated to A013918, is  defined as: $a_n$ is prime 
and sum of all primes $\le a_n$ is prime. 
The first few terms of this sequence are:  $2,3,7,13,37,43,281$. 
  \end{remark}
As a further  application of Conjecture 3.3, here we 
obtain the asymptotic expression for the $k$th prime in the sequence $(S_n)$. 

   \begin{corollary}[The asymptotic expression for 
the $k$th prime in the sequence $(S_n)$]
Let $q_k$ $(k=1,2,\ldots)$ be the $k$th  prime  in the 
sequence $(S_n)$. Then 
  \begin{equation}\label{(30)}
q_k\sim 2k^2\log^3 k\,\, {\rm as \,\,}k\to\infty.
   \end{equation}
 \end{corollary} 
\begin{proof} If for a pair $(k,n)$ there holds $q_k=S_n$, then
by Conjecture 3.3, we have
   \begin{equation}\label{(31)}
k\sim \frac{n}{\log n}\quad {\rm as}\,\, n\to\infty,
   \end{equation}
so that $n\sim k\log n$, and hence $\log n\sim \log k$ as $n\to\infty$.
Inserting this into \eqref{(23)}, we find that 
      \begin{equation*}
q_k=S_n\sim 2n^2\log n\sim 2(k\log n)^2\log n=
2k^2\log^3 n \sim 2k^2\log^3 k,
 \end{equation*} 
as desired.
   \end{proof}
   \begin{corollary}
Let $q_k$ $(k=1,2,\ldots)$ be the $k$th  prime   in the sequence
$(S_n)$. Then 
  \begin{equation}\label{(32)}
q_k\sim 2p_k^2\log k \,\, {\rm as \,\,}k\to\infty
   \end{equation}
and 
 \begin{equation}\label{(33)}
q_k\sim p_{k^2}\log^2k  \,\, {\rm as \,\,}k\to\infty.
   \end{equation}
 \end{corollary} 
  \begin{proof}
From  (30) and the fact that $p_k\sim k\log k$ we find that
    $$
q_k\sim 2(k\log k)^2\log k\sim 2p_k^2\log k,
    $$
which proves (32).

Similarly, from  (30) and 
$p_{k^2}\sim k^2\log k^2= 2k^2\log k$ we find that
    $$
q_k\sim (k^2\log k^2)\log^2k \sim p_{k^2}\log^2k,
    $$ 
which implies (33).
 \end{proof}

Furthermore, we have the following result.

   \begin{corollary}
Let $q_k$ be the $k$th prime in the sequence $(S_n)$ with $q_k=S_n$. Then
   \begin{equation}\label{(34)}
\lim_{k\to\infty}\frac{k\log k}{n}=1.
    \end{equation}
   \end{corollary}
\begin{proof} The asymptotic relation (31) implies that 
$\log n/\log k\sim 1$, which substituting in (31) immediately 
gives (34).
 \end{proof}
Motivated by some heuristic arguments and  computations for some small integer 
values $d$, we propose the following generalization of Conjecture 3.3.
    \begin{conjecture}
For  any fixed nonnegative integer $d$  
the sequence $(S_n^{(d)})_{n=1}^{\infty}$ defined as 
 $$
S_n^{(d)}=2d+S_n=2d+\sum_{i=1}^{2n}p_i,\quad n=1,2,\ldots
 $$ 
satisfies the Restricted Prime Number Theorem. In other words,  
as $n\to\infty$, 
 \begin{equation}\label{(35)}\begin{split}
\pi_n^{(d)}:=&\pi_{(2d+S_k)}(2d+S_n)=\# \{p:\, p\,\,{\rm is\,\,a\,\, prime\,\,
and}\,\, p=2d+S_i \\
&{\rm for\,\,some\,\,} i \,\,{\rm with\,\,} 1\le i\le n\}
\sim \frac{n}{\log n}.
\end{split}\end{equation}
For $d=0$ this conjecture is in fact Conjecture {\rm 3.3}  
{\rm(}cf. Sloane's sequence {\rm A013918} mentioned above{\rm )}.  
  \end{conjecture}

\begin{remark}  Conjecture 3.3 and the fact that by (23) $S_n\sim 2n^2\log n$
imply that the average difference between consecutive primes in the 
sequence $(S_n)$ near to $2n^2$ is approximately $\log (2n^2)\sim 2\log n$.
 \end{remark}

\begin{remark} 
 Numerous computational results concerning the sums of the first $n$ primes 
(partial sums of consecutive primes) given by the Sloane's sequence A007504  
(here denoted as $S_n'$), and certain their curious arithmetical properties
are presented in the following Sloane's sequences in OEIS \cite{sl}: A051838, 
A116536, A067110, A067111, A045345, A114216, A024011, A077023, A033997, 
A071089, A083186, A166448, A196527, A065595, A165906, A061568, 
A066039, A077022, A110997, A112997, 
A156778, A167214, A038346, A038347, A054972, A072476, 
A076570, A076873, A077354,  A110996, A123119, A189072,
A196528, A022094, A024447, A121756, A143121, A117842, 
A118219, A131740, A143215, A161436, A161490,  A013918 etc.
   \end{remark}
  Since the sequence $(S_n)$ is a subsequence of 
the sequence $(S_n')$ with $S_n'=\sum_{k=1}^np_k$  whose all 
terms with odd indices $n$ 
are even integers, it follows that in accordance to Definition 1.1, 
Conjecture 3.3 is equivalent to 
    $$
\omega_{(S_k')}(n)\sim \frac{n}{2\log n}.
   $$ 
Therefore, Conjecture 3.3 is equivalent with the following one.

\vspace{2mm}
\noindent{\bf Conjecture 3.3'.}  {\it Let  $(S_n')$ 
be a sequence defined as $S_n'=\sum_{k=1}^np_k$, $n=1,2,\ldots$.
Then} 
       \begin{equation}\label{(36)}
\omega_{(S_k')}(x)=\frac{x}{2\log x}\quad for \quad x\in(1,\infty).
       \end{equation}
 \begin{proposition} For each $n\ge 3$ we have
     \begin{equation}\label{(37)}
1\le \frac{S_{n}}{2n^2\log n}<1+\frac{\log 2}{\log n}+
\frac{\log\log (2n)}{\log n}.
    \end{equation} 
\end{proposition}
 \begin{proof}
 By {\it Mandl's inequality} (see, e.g., \cite{rs2}, \cite{du1}),
for each $n\ge 9$ there holds
       \begin{equation}\label{(38)}
S_n'<\frac{n}{2}p_n
       \end{equation} 
 (for a refinement of (38), see \cite[the inequality 2.4]{ha1}). 
Mandl's inequality (38) with $2n$ instead of $n$ becomes $S_n<np_{2n}$ with 
$n\ge 5$. This inequality together with the known inequality 
(see, e.g., \cite[p. 69]{rs1}) 
 \begin{equation*}
p_{2n}<2n(\log  n+\log 2+\log\log (2n))
\quad{\rm for\,\, all}\,\, n\ge 3
 \end{equation*}
immediately yields 
  \begin{equation}\label{(39)}
S_{n}<2n^2(\log n+\log 2+\log\log(2 n))
\quad{\rm for\,\, all}\,\, n\ge 5.
  \end{equation}
On the other hand, a lower bound for $S_n'$ can be obtained by using
{\it Robin's inequality} (see, e.g., \cite[p. 51]{du1}) which 
asserts that for every $n\ge 2$
 \begin{equation}\label{(40)}
np_{[n/2]}\le S_n'.
      \end{equation} 
  The inequality (40) with $2n$ instead of $n$ and 
the inequality  $n\log n\le p_{n}$ 
with $n\ge 3$ (see, e.g., \cite[p. 69]{rs1}) yield
   \begin{equation}\label{(41)}
2n^2\log n\le S_n \quad{\rm for}\,\, n\ge 3. 
  \end{equation}
The inequalities (39) and (41) immediately yield 
   \begin{equation*}
\log n\le \frac{S_{n}}{2n^2}<\log n+\log 2+\log \log(2 n)
\quad{\rm for\,\, all}\,\, n\ge 5,
    \end{equation*}
or equivalently,
     \begin{equation}\label{(42)}
1\le \frac{S_{n}}{2n^2\log n}<1+\frac{\log 2}{\log n}+
\frac{\log \log (2 n)}{\log n}
\quad{\rm for\,\, all}\,\, n\ge 5.
    \end{equation} 
The inequalities given by (42) coincide with these of (37) for 
$n\ge 5$. A direct calculation chows that (37) is also satisified 
for $n=3$ and $n=4$. This  completes the proof.
\end{proof}
 \begin{corollary}
Let $q_k=S_m$ be the $k$th prime in the sequence $(S_n)_{n=1}^{\infty}$.
Then for all $k\ge 3$ there holds 
     \begin{equation}\label{(43)}
2m^2\log m< q_k<2m^2(\log m+\log(\log (2m)+\log 2).
    \end{equation} 
  \end{corollary}
\begin{proof}
The above inequalities coincide with (37) of Proposition 3.12 
with $n=m$ and $q_k=S_m$.
 \end{proof}
   \begin{remark}  Z.-W. Sun \cite[the case $\alpha=1$ 
in Lemma 3.1]{s2} showed that for all $n\ge 2$
  $$
S_n'>2+\frac{n^2\log n}{2}\left(1-\frac{1}{2\log n}\right),
  $$
which with $2n$ instead of $n$ becomes
  $$
S_n>2+2n^2\left(\log n +\log\frac{2}{\sqrt{e}}\right)\approx 
2+2n^2(\log n +0.193147),
  $$
whence it follows that 
  $$
\frac{S_n}{2n^2\log n}>1+\frac{0.193147}{n}.
  $$ 
The above inequality is stronger that the left hand side of the inequality
 (37). Accordingly, if $q_k=S_m$, then 
 the first  inequality of (43) can be refined in the form 
 $$
q_k>2+2m^2(\log m +0.193147)\quad {\rm for\,\, all\,\,} k\ge 3.
  $$

On the other hand, combining  the inequalities 
(46) and (47) from the next section 
with the inequalities $S_n>2np_{n}$ and $S_n<np_{2n}$ 
(given in proof of Proposition 3.12), respectively,
we immediately obtain the following refinement of 
Proposition 3.12.
\end{remark}

\begin{proposition}
 For each $n\ge 3$ there holds
 \begin{equation*}
\frac{S_n}{2n^2\log n}\ge 
1+\frac{\log\log n-1}{\log n}+\frac{\log\log n-2.2}{\log^2 n},
  \end{equation*}
and for each $n\ge 344192$, we have
 \begin{equation*}
\frac{S_n}{2n^2\log n}\le 
1+\frac{\log\log (2n)+\log 2-1}{\log n}+\frac{\log\log (2n)-2}{(\log n)
\log (2n)}.
  \end{equation*}
\end{proposition}
   \begin{remark} 
If $q_k=S_m$, then in view of the first inequality 
of Proposition 3.15, the first  inequality of (43)
may  be replaced by the following one:
 $$
q_k>2m^2\left(\log m + \log\log m-1+\frac{\log\log m-2.2}{\log m}\right)
\quad {\rm for\,\, all\,\,} k\ge 3.
  $$
\end{remark} 
     
   \begin{remark} 
The inequalities (38) and (40) and the asymptotic expression 
$p_n\sim n\log n$ show that the average of the first $n$ primes
is asymptotically equal to $(n\log n)/2$ (cf.  Sloane's sequence 
A060620 in \cite{sl}), that is, 
   \begin{equation*}
\frac{S_n'}{n}\sim \frac{n\log n}{2}\quad{\rm as}\quad n\to\infty.
    \end{equation*}  
  \end{remark} 
Conjecture 3.3 suggests the fact that for the sequence $(S_n)$
would be valid the analogues of some other classical results 
and conjectures closely related to the Prime Number Theorem 
and Riemann Hypothesis. 
In particular, 
if  $Q=\{q_1,q_2,\ldots,q_k,\ldots\}$ 
is a set of all primes $q_1<q_2<\ldots <q_k<\cdots$
in the sequence $(S_n)$,
it can be of interest to establish the asymptotic expression for $q_k$ as 
$k\to\infty$.

Finally, heuristic arguments, some computational results
and  Conjecture
3.3 lead to the follwing its generalization (cf. 
Sloane's sequence  A143121 - triangle read by rows,
$T(n,k)=\sum_{j=k}^np_j$, $1\le k\le n$; see the columns 
in Example of this sequence).

    \begin{conjecture}
For any fixed positive integer $k$, let 
$(S_n^{(k)}):=(S_n^{(k)})_{n=1}^{\infty}$ be the sequence whose $n$th term is 
defined as
       $$
S_n^{(k)}=\sum_{i=1}^{2n+1}p_{i+k},\quad n\in\Bbb N.
       $$
Then the sequence  $(S_n^{(k)})$ satisfies the Restricted 
Prime Number Theorem. 
  \end{conjecture}

For example, there are 78498 (resp. 664579) primes less than $10^6$
(resp. $10^7$),
while the computations show that  among  the first  $10^6$ (resp. $10^7$) 
terms of  the sequences $(S_n)$, $(S_n^{(k)})$ with $k=1,2,\ldots,12$ 
 there are  
69251 (resp. 594851), 69581 (resp. 594377),   68844 (resp. 593632), 
68883 (resp. 595733), 
69602 (resp. 596609),   69540 (resp. 596558), 69414 (resp. 595539), 
69317 (resp. 594626), 69455 (resp. 595474), 
69268 (resp. 594542), 68891 (resp. 593807), 69251 (resp. 594383), 
69564 (resp. 595270) 
primes, respectively.     

\section{The asymptotic expression for the $k$th prime in the sequence $(S_n)$}

As an easy consequence of the Prime Number Theorem,  it can be deduced that 
$p_n\sim n\log n$ as $n\to\infty$  (see, e.g., \cite{mr}).
Furthermore, a particular asymptotic expansion  for $p_n$
(see \cite{mr} or \cite[the equality (66) of Section 6]{rt};
also see Sloane's sequence A200265)
yields 
  \begin{equation}\label{(44)}
p_n=n\left(\log n+\log\log n+O\left(\frac{\log\log n}{\log n}\right)\right).
  \end{equation}
It is also known that (see \cite{du2} and \cite[p. 69]{rs1})
 \begin{equation}\label{(45)}
n(\log n+\log\log n-1)<p_n<n(\log n+\log\log n).
 \end{equation}
   A more precise work about this can be found in \cite{rob2} and 
\cite{sa} where related results are as follows: 
 \begin{equation}\label{(46)}
n\left(\log n+\log\log n-1+\frac{\log\log n-2.2}{\log n}\right)\le p_n
\quad {\rm for}\,\,n\ge 3
 \end{equation}
and 

 \begin{equation}\label{(47)}
 p_n\le n\left(\log n+\log\log n-1+\frac{\log\log n-2}{\log n}\right)\quad {\rm for}
 \,\,n\ge 688383.
  \end{equation}
The inequalities (46) and (47) immediately yield the following result.

 \begin{corollary}
For each $n\ge 688383$ the interval
 \begin{equation}\label{(48)}
\left[n\left(\log n+\log\log n-1+\frac{\log\log n-2.2}{\log n}\right), 
n\left(\log n+\log\log n-1+\frac{\log\log n-2}{\log n}\right)\right]
 \end{equation}
contains at least one prime. Furthermore, the lenght $l_n$ of this
interval is
 \begin{equation}\label{(49)}
l_n = \frac{0.2n}{\log n}\sim \frac{0.2p_n}{\log^2 n}. 
  \end{equation} 
 \end{corollary}

As an application of Conjecture 3.3, we obtain the following result.
   \begin{corollary}
Let $q_k$ be the $k$th prime in the sequence $(S_n)$ with 
$S_n=\sum_{k=1}^{2n}p_k$. If $q_k=S_m$, then under Conjecture {\rm 3.3} 
there holds
  \begin{equation}\label{(50)}
    q_k\sim 2k^2\log^3 k\sim 2m^2\log m\sim p_{\lfloor k^2\log^2 k\rfloor}
{\rm \quad as}\,\, k\to\infty,
    \end{equation}
and
    \begin{equation}\label{(51)}
\lim_{k\to\infty}\frac{k\log k}{m}=1.
    \end{equation}
   \end{corollary}
\begin{proof}
The first asymptotic relation of (50) coincides with (30) of Corollary 
3.6. Further, by (37) of Proposition 3.12, we have 
 \begin{equation}\label{(51)}
q_k=S_m\sim 2m^2\log m.
  \end{equation}
Moreover, we have 
  \begin{equation}\label{(53)}
p_{[k^2\log^2 k]}\sim k^2(\log^2 k)\log (k^2\log^2 k)\sim
2k^2\log^3 k {\rm \quad as}\,\, k\to\infty.
   \end{equation}
The last two asymptotic expressions of (50) follow from 
(52) and (53).

It remains to prove (51). If we suppose that (51) is not satisfied, then there 
exists $\varepsilon>0$ and an infinite subsequence $(k_j,m_j)_{j=1}^{\infty}$ of the 
sequence $(k,m)_{k=1}^\infty$
such that $k_j\log k_j\ge (1+\varepsilon)m_j$ for all $j\in \Bbb N$ 
or   $k_j\log k_j\le (1-\varepsilon)m_j$ for all $j\in \Bbb N$. In 
the first case,  using (50) for all sufficiently large $j$, we find that 
  $$
m_j^2\log m_j\sim k_j^2\log^3 k_j\ge (1+\varepsilon)^2m_j^2\log k_j,
 $$
whence we immediately get $\log m_j\ge (1+\varepsilon)^2\log k_j$,
or equivalently, $m_j\ge k_j^t$ with $t=(1+\varepsilon)^2>1$. 
From the previous inequality and the fact that $t>2$ we have
  $$
m_j^2\log m_j>m_j^2\ge k_j^{2t}\gg k_j^2\log^3 k_j,
  $$  
which contradicts the fact that by (50) $2k^2\log^3 k\sim 2m^2\log m$.
In a similar way as in the first case,  in the second case we find that
$m_j\le k_j^s$ for a constant $s=(1-\varepsilon)^2<1$.
Then choosing  a sufficiently large $j_0$  such that $\log m_j<m_j^{1/s-1}$ 
for all $j>j_0$, in view of the fact that $s+1<2$ we get
    $$
m_j^2\log m_j< m_j^{1+1/s}\le k_j^{s+1}\ll k_j^2\log^3 k_j.
  $$  
A contradiction, and therefore, (51) is true.
\end{proof}
\begin{remark} 
From (50) and $p_k\sim k\log k$ we see that
   \begin{equation}\label{(54)}
\frac{q_k}{2k\log^2 k}\sim k\log k\sim p_k\quad{\rm as}\quad k\to\infty.
 \end{equation}
The above asymptotic expression together with the assumption that Conjecture 
3.3 is true suggests the fact that for the sequence 
$\left(q_k/(k\log^2 k\right))$ would be satisfied the asymptotic expansion
similar to (44), i.e., 
     \begin{equation}\label{(55)}
\frac{q_k}{2k\log^2 k}= k(\log k+\log\log k+Q_k),
    \end{equation}
with some sequence $(Q_k)$. 
Motivated by (55), we establish the following asymptotic expression for 
the $k$th prime in the sequence $(S_n)$.
\end{remark}

   \begin{theorem}[The asymptotic expression for 
the $k$th prime in the sequence $(S_n)$]
Let $q_k$ be the $k$th  prime in the sequence
$(S_n)$ $(k=2,3,\ldots)$. Then under Conjecture {\rm 3.3}
  there exists a sequence $(M_k)$  of positive rael numbers 
such that $\lim_{k\to\infty}M_k=1$ and  
    \begin{equation}\label{(56)}
q_k= 2M_k^5k^2\log^2 k(\log k+\log\log k+2\log M_k).
   \end{equation}
\end{theorem}
For the proof of Theorem 4.4 we will need the following result.
  \begin{lemma} 
Let $S_m=q_k$ be the $k$th prime in the sequence $(S_n)$. Then under 
Conjecture $(3,3)$ we have 
  \begin{equation}\label{(57)}
q_k\sim \frac{2m^2\sqrt{m}\log m}{\sqrt{k\log k}}\quad as\,\, k\to\infty.
  \end{equation}
\end{lemma}
\begin{proof} 
First notice that, under notations of Lemma 4.5, Conjecture 3.3 yields 
(cf. (51) of Corollary 4.2)
 \begin{equation}\label{(58)}
k\sim \frac{m}{\log m}\quad as\,\, k\to\infty.
  \end{equation}
Using (58), we find that
   \begin{equation}\label{(59)}
 \frac{2m^2\sqrt{m}\log m}{\sqrt{k\log k}}\sim 
 \frac{2m^2\sqrt{m}\log m}{\sqrt{m\left(1-\frac{\log\log m}{\log m}\right)}}
\sim 2m^2\log m \quad as\,\, k\to\infty.
  \end{equation}
The asymptotic relation (59) and the fact that by (50) of Corollary 4.2,
$q_k=S_m\sim 2m^2\log m$ immediately yield (57).
  \end{proof} 
\begin{proof}[Proof of Theorem $4.4$] 
Let $(C_k)_{k=2}^{\infty}$ be a sequence of positive real numbers such that
$m(k)=m=C_kk\log k$ with $k\ge 2$ and $q_k=S_m$. Then by (51)
of Corollary 4.2, we have $C_k\to 1$ as $k\to\infty$. 
Taking $m=C_kk\log k$ into (57) of 
Lemma 4.5,  as $k\to\infty$ we obtain  that
 \begin{equation}\label{(60)}\begin{split}
q_k &\sim \frac{2\sqrt{C_k^5k^5\log^5 k}(\log k+\log\log k +
\log C_k)}{\sqrt{k\log k}}\\
&=2C_k^2\sqrt{C_k}k^2\log^2 k(\log k+\log\log k+\log C_k)=:f(k,C_k).
  \end{split}\end{equation}
Let $(\delta_k)$ be a positive real sequence such that 
  \begin{equation}\label{(61)}
q_k=\delta_k f(k,C_k)\quad {\rm for\,\, each\,\,} k\ge 2.
   \end{equation}
Then from (60) we see that $\delta_k\to 1$ as $k\to\infty$.
For a fixed $k\ge 2$ consider the equation $f_k(x)=\delta_k f(k,C_k)$ which 
can be written in the form
   \begin{equation}\label{(62)}      
x^2\sqrt{x}(\log k+\log\log k+\log x)=
\delta_k C_k^2\sqrt{C_k}(\log k+\log\log k+\log C_k).
     \end{equation}
 Notice that for any fixed integer $k\ge 2$, the real function
  $f_k(x)$  defined as
  $$    
f_k(x)=2x^2\sqrt{x}(\log k+\log\log k+\log x), \quad x>0, 
   $$
satisfies the limit relations $\lim_{x\to +\infty}f_k(x)= +\infty$
and  $\lim_{x\to +0}f_k(x)=0$. From this  
  it can be easily shown that for each integer $k\ge 2$ the  equation (62)
   has a positive real solution $x_k$. Using the facts that 
$\lim_{x\to +\infty}C_k=\lim_{x\to +\infty}\delta_k$=1, it can be easily show 
that $\lim_{k\to\infty}x_k=1$. Then taking $x_k=M_k^2$  
($k=2,3,\ldots$), then $\lim_{k\to\infty}M_k=1$ and 
 by (62) we find that $f_k(M_k^2)=\delta_kf(k,C_k)=q_k$, 
whence it follows that 
     $$
q_k= 2M_k^5k^2\log^2 k(\log k+\log\log k+2\log M_k). 
     $$
This proves (56) and the proof is completed. 
  \end{proof}

Computational results (cf. the eighth column of   Table 1 of Section 6) suggest the additional 
relationship between $k$'s  and $m$'s  as follows.
 \begin{conjecture}
For each pair $(k,m)$ with $k\ge 1$ and $q_k=S_m$ we have
  \begin{equation}\label{(63)}
\lfloor k\log k \rfloor +1\le m, 
   \end{equation}
or equivalently,
   \begin{equation}\label{(64)}
q_k\ge S_{\lfloor k\log k \rfloor +1}. 
   \end{equation}
Furthermore, for each $k\ge 10^4$,
   \begin{equation}\label{(65)}
m\le \lfloor 1.4 k\log k\rfloor, 
   \end{equation}
or equivalently,
   \begin{equation}\label{(66)}
q_k\le S_{\lfloor 1.4 k\log k\rfloor}. 
   \end{equation}
 \end{conjecture}
 \begin{corollary} 
If the inequality  {\rm (63)} of Conjecture {\rm 4.6} is true, then 
for each $k\ge 1$ there holds
    \begin{equation}\label{(67)}
q_k> 2k^2(\log^2k)(\log k+\log\log k).
    \end{equation}
 \end{corollary} 
\begin{proof} Combining the inequality (63) with the inequality on 
the  left hand side of (37) of Proposition 3.12, we find that 
     \begin{equation}\label{(68)}\begin{split}
q_k&=S_m\ge S_{\lfloor k\log k\rfloor +1}\ge 2(\lfloor k\log k\rfloor+1)^2
\log(\lfloor k\log k\rfloor+1)\\
&>2k^2(\log^2k)\log (k\log k)= 2k^2(\log^2k)(\log k+\log\log k),
    \end{split}\end{equation}
as desired.
 \end{proof}
 \begin{corollary} 
If the inequality  {\rm (63)} of Conjecture {\rm 4.6} is true, then 
$M_k>1$ for each $k\ge 1$, where $(M_k)$ is the sequence  defined by 
{\rm (56)} of Theorem {\rm 4.4}.
 \end{corollary} 
\begin{proof}
The assertion follows immediately from the inequality (67) and 
the expression {\rm (56)} for $q_k$ given by Theorem {\rm 4.4}.
 \end{proof}
Finally, in view of the data  of the last column in  Table 1 of Section 6 and 
some considerations presented above, we propose the following 
conjecture which is stronger than Corollary 4.7.
 \begin{conjecture} For every 
$k\ge 252028$ with $q_k=S_m$ there holds 
   \begin{equation}\label{(69)}
q_k>\frac{2m^2\sqrt{m}\log m}{\sqrt{k\log k}}.
   \end{equation}
\end{conjecture} 
In view of the well known inequality 
$p_k>k(\log k+\log\log k)$, the following conjecture 
 is also stronger than Corollary 4.7.
 \begin{conjecture} There exists $k_0\in \Bbb N$ such that for every 
$k\ge k_0$  there holds 
   \begin{equation*}
q_k>2kp_k\log^2k,
   \end{equation*}
where $p_k$ is the $k$th prime.
\end{conjecture} 

 \begin{remark}
The last column of Table 1 presented in Section 6 shows that 
    \begin{equation*}
q_k\approx \frac{2m^2\sqrt{m}\log m}{\sqrt{k\log k}}
   \end{equation*}
is a ``good'' approximation for the $k$th prime sum $q_k$.
Notice that this approximation can be written as 
     \begin{equation*}
q_k\approx 2m^2\log m \cdot\sqrt{\frac{m}{k\log k}},
   \end{equation*}
where the values $\sqrt{m/(k\log k)}$ slowly tend to 1 as $k$ grows.
In particular, from the last row of Table 1 of Section 6 we see that  
for $m=10^9-2$ (i.e., for $k=46388006$) we have  
$\sqrt{m/(k\log k)}\approx 1.105079$. Hence, in view of the above 
approximation, we believe that for 
all values
$m$ up to $10^9$ there holds 
 $$
q_k>2.2m^2\log m.
 $$ 
Notice that some values of the 
sequences $(Q_k')$ such that 
 $$
Q_k'=\frac{q_k-2k^2(\log^2 k)(\log k+\log\log k)}{2k^2\log^2 k\log\log k}
$$  
and the sequence $(Q_k'')$ such that  
$$
Q_k''=\frac{q_k-2(p_k)^2\log k}{2k^2\log^2 k\log\log k}
 $$
are given in Table 3 of Section 6.  Table 3  also shows that for 
almost all values $m$ up to $10^9$ (i.e., for $k\le 46388006$) 
there holds 
 $$
kp_k>1.1m^2\log m.
  $$
  \end{remark}

Finally, Table 2 of Section 6 leads to the following conjecture 
whose both parts are obviously stronger than Conjecture 4.6.

\begin{conjecture} Let $\pi(x)$ be the prime counting function,
and let $\pi_n$ be the number of primes in the set 
$\{S_1,S_2,\ldots,S_n\}$.  Then 
       $$
     \pi_n<\pi(n)
        $$  
 for for each $n\ge 10^4$ and 
    $$
\pi_n<\frac{n}{\log n}
     $$      
for each $n\ge 10^5$.
\end{conjecture}

\section{Estimations of values $M_k$ from Theorem 4.4}

The computational results related to the search of primes 
in the sequence $(S_n)$ given in the following section
(Table 1) and some heuristic arguments suggest the fact that 
the sequence $(S_n/(2n^2 \log n))_{n=2}^{\infty}$ plays an important 
role for estimating the values $M_k$ $(k=1,2,\ldots)$ in the expression 
(56) for the $k$th prime $q_k$ in the sequence $(S_n)$.    

Here we first consider the sequence $(S_n/(2n^2))$.

\begin{proposition} The sequence $(v_n)$ defined as
 \begin{equation}\label{(70)}
v_n=\frac{S_n}{2n^2},\quad n\in \Bbb N,
  \end{equation}
is increasing for $n\ge 2$. 
  \end{proposition}
\begin{proof}
Since $S_{n+1}=S_n+p_{2n+1}+p_{2n+2}$, an easy 
calculation shows that $r_n<r_{n+1}$ is equivalent with
   $$
\frac{S_n}{2n^2}<\frac{p_{2n+1}+p_{2n+2}}{2(2n+1)},
  $$
which can be written as 
    \begin{equation}\label{(71)}
\frac{p_{2n+1}+p_{2n+2}}{2}>S_n\left(\frac{1}{n}+\frac{1}{2n^2}\right).
  \end{equation}
By a refinement of Mandl's inequality due to Hassani \cite{ha1},
for every $n\ge 10$ we have
    \begin{equation}\label{(72)}
\frac{n}{2}p_n-\sum_{i=1}^np_i>0.01659n^2.
   \end{equation}
Replacing $n$ by $2n$ into (72) it  becomes
 \begin{equation}\label{(73)}
p_{2n}-\frac{S_n}{n}>0.06636n^2  \quad {\rm for\,\, all\,\,} n\ge 5. 
   \end{equation}
Further, by the inequality (37) of Proposition 3.12 we have
  \begin{equation}\label{(74)}
\log(2n)+\log\log(2n)>\frac{S_n}{2n^2}\quad {\rm for\,\, all\,\,} n\ge 5. 
   \end{equation}
By using  {\tt Mathematica 8}, it is easy to to prove the inequality 
  \begin{equation}\label{(75)}
0.06636n^2> \log(2n)+\log\log(2n) \quad {\rm for\,\, all\,\,} n\ge 8. 
   \end{equation}
Finally, combining the inequalities (73), (74), (75)  and
the obvious inequality $(p_{2n+1}+p_{2n+2})/2>p_{2n}$ immediately
gives (71) for all $n\ge 8$. This together with a direct verification 
that $v_n<v_{n+1}$ for $2\le n\le 8$ concludes the proof.  
  \end{proof}
  \begin{remark} 
Notice that the sequence $(v_n)$ defined by (70) is a subsequence 
of the sequence $(v_n')$ defined as 
  \begin{equation*}
v_n'=\frac{2S_n'}{n^2}:=
\frac{2\sum_{i=1}^np_i}{n^2},\quad n\in \Bbb N;
  \end{equation*}
 namely, $v_n=v_{2n}'$ for all $n=1,2,\ldots$. 
Similarly as in the proof of Proposition 5.1, it can be shown 
 that the sequence $(v_n')$ is  increasing for $n\ge 4$.
 \end{remark}
  
Contrary to Proposition 5.1, we propose the following conjecture.

\begin{conjecture} The sequence $(t_n)$ defined as
 \begin{equation}\label{(76)}
t_n=\frac{S_n}{2n^2\log n},\quad n\in \Bbb N\setminus \{1\},
  \end{equation}
is decreasing on the range  $\{n\in\Bbb N:\, n\ge 1100\}$ 
{\rm(}$m=1099$ is a maximal value between total $40$ values up to
$n=200000$  for which $t_{m+1}>t_{m}${\rm )}.
  \end{conjecture}
  \begin{remark} 
Notice that the sequence $(t_n)$ defined by (76) is a subsequence 
of the sequence $(t_n')$ defined as 
  \begin{equation}\label{(70)}
t_n'=\frac{2S_n'}{n^2\log (n/2)}:=
\frac{2\sum_{i=1}^np_i}{n^2\log (n/2)},\quad n\in \Bbb N;
  \end{equation}
 namely, $t_n=t_{2n}'$ for all $n=1,2,\ldots$. 
We conjecture that the sequence $(t_n')$ is  decreasing on the range  
$\{n\in\Bbb N:\, n\ge  2199\}$ 
($m=2198$  is a maximal value  up to 
$n=10^6$   for which $t_{m+1}'>t_{m}'$).
 \end{remark}

\begin{corollary} Let  $(t_n)$ be the sequence defined in  Conjecture 
{\rm 5.3}. Then  under Conjecture 
{\rm 5.3}, for each $n\ge 1100$ there holds
   \begin{equation}\label{(71)}
t_{n+1}<t_{n}<t_{n+1}\left(1+\frac{1}{n\log n}\right).
   \end{equation}
 \end{corollary}
\begin{proof} Proposition 5.1, Conjecture 5.3 and the well known
inequality $(1+1/n)^n<e$ with $n\ge 1$ immediately imply that 
for all $n\ge 1100$ there holds
  \begin{equation}\label{(72)}\begin{split}
0<t_n-t_{n+1}&=\frac{S_n}{2n^2\log n}-\frac{S_{n+1}}{2(n+1)^2\log (n+1)}\\
&<\frac{S_{n+1}}{2(n+1)^2\log n}-\frac{S_{n+1}}{2(n+1)^2\log (n+1)}\\
&=\frac{S_{n+1}}{2(n+1)^2}\cdot
\frac{\log\left( 1+\frac{1}{n}\right)^n}{n(\log n)
(\log (n+1))}\\
&<\frac{S_{n+1}}{2(n+1)^2\log (n+1)}\cdot\frac{1}{n\log n}\\
&=\frac{t_{n+1}}{n\log n}.
   \end{split}\end{equation}
From (79) we immediately get (78). 
  \end{proof}
 \begin{corollary} Let  $(t_n)$ be the sequence defined in  Conjecture 
{\rm 5.3}. Then  under Conjecture {\rm 5.3}, for each $n\ge 1101$ we have 
   \begin{equation}\label{(80)}
t_n>\frac{17}{8\log 2}\left(\left(1+\frac{1}{2\log 2}\right)
\left(1+\frac{1}{3\log 3}\right)\cdots \left(1+\frac{1}{(n-1)\log (n-1)}
\right)\right)^{-1}
   \end{equation}
and 
  \begin{equation}\label{(81)}
S_n>\frac{17n^2\log n}{4\log 2}\left(\left(1+\frac{1}{2\log 2}\right)
\left(1+\frac{1}{3\log 3}\right)\cdots \left(1+\frac{1}{(n-1)\log (n-1)}
\right)\right)^{-1}.
   \end{equation}
\end{corollary}
 \begin{proof} 
By the right hand side of the inequality (78), 
we obtain that for each $n\ge 1101$
    \begin{equation}\label{(82)}
t_n>t_{n-1}\left(1+\frac{1}{(n-1)\log (n-1)}\right)^{-1}.
     \end{equation}
By iterating the  inequality  (82) $(n-2)$ times and  
taking $S_2=17$ in  $(n-2)$th step, we immediately 
obtain the inequality (80). Substituting $t_n=S_n/(2n^2\log n)$ into 
(80) gives the inequality (81). 
  \end{proof}
Notice that under Conjecture 5.3, the sequence $(t_n)$ defined by (76)
as 
   $$
t_n=\frac{S_n}{2n^2\log n},\quad n\in \Bbb N\setminus\{1\},
    $$
is decreasing on the range  $\{n\in\Bbb N:\, n\ge 1100\}$. 
As noticed above, the computational results for ``prime sums'' 
given in Table 1 of Section 6  suggest the fact that 
the sequence $(t_n)_{n=2}^{\infty}$ plays an important 
role for estimating the values $M_k$ $(k=1,2,\ldots)$ in the expression 
(56) for the $k$th prime $q_k$ in the sequence $(S_n)$.    
Accordingly, we propose the following two conjectures 
concerning the upper and lower bounds of the sequence $(M_k)$.

  \begin{conjecture}[The upper bound of the sequence $(M_k)$]
Let  $(M_k)_{k=1}^\infty$ be the sequence defined by by the expression 
{\rm (56)} of Theorem {\rm 4.4}. Then 
   \begin{equation}\label{(83)}
 M_k\le t_k= \frac{S_k}{2k^2\log k}:=M_k^{(u)}\quad {\rm for\,\,all}\,\,
k\ge 2. 
   \end{equation}
 \end{conjecture}

  \begin{corollary}
Let  $(M_k)_{k=1}^\infty$ be the sequence defined by the expression 
{\rm (56)} of Theorem {\rm 4.4}. Then under Conjecture {\rm 5.7}
there holds
   \begin{equation}\label{(84)}
M_k\le 1+\frac{\log 2+\log\log (2 k)}{\log k}\quad {\rm for\,\,all}\,\,
k\ge 2. 
   \end{equation}
 \end{corollary}

  \begin{proof}
Combining the inequality on the right hand side of  (37) from Proposition 3.12
with the inequality (83), we immediately obtain (84).  
  \end{proof}
  \begin{corollary}
Let $q_k$ be the $k$th  prime in the sequence
$(S_n)$ $(k=1,2,\ldots)$. Then under Conjectures {\rm 3.3}
and {\rm 5.7} for all $k\ge 2$ there holds  
     \begin{equation}\label{(85)}
q_k<2k^2\log^3k
\left(1+\frac{\log 2+\log\log (2 k)}{\log k}\right)^5
\left(1+ \frac{\log\log k}{\log k}+\frac{2\log 2+2\log\log (2 k)}{\log^2 k}
\right). 
   \end{equation}
  \end{corollary}
  \begin{proof}
Applying  the inequality $\log(1+x)<x$ with $x>0$ to (84), we find that 
    \begin{equation}\label{(86)}
\log M_k\le \frac{\log 2+\log\log (2 k)}{\log k}\quad {\rm for\,\,all}\,\,
k\ge 2. 
   \end{equation}
Inserting the inequalities (84) and (86) into the expression (56) of 
Theorem 4.4 for $q_k$, we immediately obtain (85).
  \end{proof}
  \begin{corollary}
Let $q_k$ be the $k$th  prime in the sequence
$(S_n)$ $(k=1,2,\ldots)$. Then under Conjectures {\rm 3.3}
and {\rm 5.7} there holds  
     \begin{equation}\label{(87)}
q_k=2k^2\log^2k(\log k+O\left(\log\log k)\right), 
   \end{equation}
or equivalently,
      \begin{equation}\label{(88)}
\frac{q_k}{2k^2\log^3k}=1+O\left(\frac{\log\log k}{\log k}\right). 
   \end{equation}
 \end{corollary}
 \begin{proof} The inequality (85) immediately yields 
 the asymptotic expression (87). 
 \end{proof}
Corollary 5.10 can be refined as follows.
  \begin{corollary}
Let $q_k$ be the $k$th  prime in the sequence
$(S_n)$ $(k=1,2,\ldots)$. Then 
under  Conjectures {\rm 3.3} and {\rm 5.7} there exists an absolute 
positive constant $C$ with $1\le C\le 6$  such that
     \begin{equation*}
q_k=2k^2\log^2k(\log k+C\log\log k+o(\log\log k)). 
   \end{equation*}
 \end{corollary}
 \begin{proof}
Using  the binomial expansion, we find that 
      \begin{equation*}
\left(1+\frac{\log 2+\log\log (2 k)}{\log k}\right)^5
= 1+ \frac{5\log\log k}{\log k}+o\left(\frac{\log\log k}{\log k}\right),
  \end{equation*}
which substituting in (85) immediately yields the estimation 
from Corollary 5.11.
  \end{proof}
 \begin{remark}
The determination of a constanst $C$ from Corollary 5.11 
is closely related to the sequence $(Q_k)$ with 
$Q_k=(q_k-2k^2\log^3 k)/(2k^2(\log^2 k)\log\log k)$,
whose values are presented in Table 3 of Section 6. 
Related data from Table 3 
and the additional computations suggest that 
    \begin{equation*}
q_k<2k^2\log^2k(\log k+6\log\log k)\quad {\rm for\,\, all\,\,}
k\ge 5\cdot 10^6. 
   \end{equation*}
\end{remark}
  \begin{conjecture}[A refined  upper  bound of the sequence $(M_k)$] 
Let  $(M_k)_{k=1}^\infty$ be the sequence defined by the expression 
{\rm (56)} of Theorem {\rm 4.4}. Then 
   \begin{equation}\label{(89)}
M_k\le t_{\lfloor k\log k\rfloor}=
\frac{S_{\lfloor k\log k\rfloor}}{2(\lfloor k\log k\rfloor)^2
\log \lfloor k\log k\rfloor}:=M_k^{(l)}
\quad {\rm for\,\,all}\,\,k\ge 5\times 10^7, 
   \end{equation}
where $\lfloor k\log k\rfloor$ is the greatest  integer not exceeding  $k\log k$.
 \end{conjecture}
  \begin{corollary}
Let $(t_n)$ be the sequence  defined by 
{\rm (76)}. Then under the inequality {\rm (63)} of Conjecture  
{\rm 4.6} and Conjecture  {\rm 5.13},
for all $k\ge 2$ the interval  
 \begin{equation}\label{(90)}\begin{split}
&\Big[2k^2(\log^2 k)(\log k+\log\log k),
2k^2(\log^2 k)\Big(1+\frac{\log(2\log (2k))}{\log k}\Big)^5\times\\
&\times (\log k+\log\log k+ 2\log \Big(1+\frac{\log(2\log (2k))}{\log k}
\Big)\Big)\Big]
   \end{split}\end{equation}
contains at least one prime that belongs to the 
  sequence $(S_n)$. In particular, the prime $q_k$
belongs to the interval given by {\rm (90)}.

Furthermore, for all $k\ge 2$ the length $l_k$  of the interval 
{\rm (90)} satisfies the inequality 
  \begin{equation}\label{(91)}
l_k<62k^2(\log k)\log(k\log k)\log(2\log (2k))+
4k^2(\log k+31\log(2\log (2k))\log(2\log (2k)).
      \end{equation}
  \end{corollary}  
\begin{proof} 
The first assertion immediately follows from the inequality
(67) of Corollary 4.7 and the inequality (83) of Conjecture 5.7.
Notice that  by the inequality on the  right hand side 
of  (37) from Proposition 3.12, we find that 
  \begin{equation}\label{(92)}
t_k< 1+\frac{\log(2\log (2k))}{\log k}
\quad {\rm for \,\, all}\,\,k\ge 5.
   \end{equation}
Then the inequality (83) of Conjecture 5.7 and  the inequality (92) immediately yield 
      \begin{equation}\label{(93)}\begin{split}
q_k&=2M_k^5k^2(\log^2 k)(\log k+\log\log k+2\log M_k)\\
&\le 2t_k^5k^2(\log^2 k)(\log k+\log\log k+2\log t_k)\\
&\le 2k^2(\log^2 k)\Big(1+\frac{\log(2\log (2k))}{\log k}\Big)^5
(\log k+\log\log k+ 2\log \Big(1+\frac{\log(2\log (2k))}{\log k}\Big)).
   \end{split}\end{equation}
The inequalities (93) and (67) of Corollary 4.7 show that the interval 
defined by (90) contains the prime sum $q_k$. 

Further, using the inequality $(1+x)^5\le 1+31x$ for 
$0\le x:=\log (2\log (2k))/\log k \le 1$ and the 
inequality $\log(1+x)<x$ for $x:=\log (2\log (2k))/\log k>0$, 
the length $l_k$ of interval 
defined by (90) can be estimated as follows.
  \begin{equation}\label{(96)}\begin{split}
l_k  \le & \left(\left(1+\frac{\log(2\log (2k))}{\log k}\right)^5-1\right)
2k^2(\log^2 k)(\log k+\log\log k)\\
&+4\left(1+\frac{\log(2\log (2k))}{\log k}\right)^5k^2(\log^2 k)\log
 \left(1+\frac{\log(2\log (2k))}{\log k}\right)\\
 \le & \frac{31\log(2\log (2k))}{\log k}2k^2(\log^2 k)(\log k+\log\log k)\\
&+4\left(1+31\cdot\frac{\log(2\log (2k))}{\log k}\right)k^2(\log^2 k) 
\frac{\log(2\log (2k))}{\log k}\\
=& 62k^2(\log k)\log(k\log k)\log(2\log (2k))+
4k^2(\log k+31\log(2\log (2k))\log(2\log (2k)).
    \end{split}\end{equation}
This completes  the proof.
   \end{proof}
Nevertheless the fact that $M_k^{(l)}$ is probably the 
upper bound of $M_k$ for all $n>5\cdot 10^7$
(see Table 1), we propose the following conjecture.
  \begin{conjecture}
Let $(t_n)$ be the sequence  defined by 
{\rm (76)}. Then  for all $k\ge 2$ the interval  
     \begin{equation}\label{(90)}
[2t_{\lfloor k\log k\rfloor}^5k^2(\log^2 k)(\log k+\log\log k+2\log 
t_{\lfloor k\log k\rfloor}),2t_k^5k^2(\log^2 k)(\log k+\log\log k+2\log t_k)]
   \end{equation}
contains at least one prime sum $q_i$ from the sequence $(S_n)$. 
  \end{conjecture}

As an application of Corollary 5.14, we obtain the following 
$(S_n)$-analogue of the well known fact that the series 
$\sum_{n=1}^{\infty}1/p_n$ diverges.
 \begin{corollary} 
The series 
    \begin{equation}\label{(96)}
\sum_{k=1}^{\infty}\frac{k\log^2 k}{q_k}
      \end{equation}
diverges.
 \end{corollary}
 \begin{proof}
It is easy to see that for each $k\ge 2$ the right bound of the interval given by 
(90) is less than $288k^2\log^3 k$, and hence, by Corollary 5.14, 
$q_k<288k^2\log^3 k$ for each $k\ge 2$. Therefore, 
$k\log^2 k/q_k>1/(288k\log k)$, and hence, 
     \begin{equation*}\begin{split}
\sum_{k=1}^{n}\frac{k\log^2 k}{q_k}
&>\frac{1}{288}\sum_{k=2}^{n}\frac{1}{k\log k}\sim \int_2^n
\frac{d\,x}{x\log x}\\
&=\log\log x \Big|_2^{n}=\log\log n-\log\log 2\to\infty\,\,  
{\rm as\,\,} n\to\infty.
      \end{split}\end{equation*}
Therefore, the series (96) diverges.    
  \end{proof}

 On the other hand, we have the following consequence of 
of Corollary 5.14.

 \begin{corollary} 
For every $\varepsilon >0$  the series 
    \begin{equation}\label{(97)}
\sum_{k=1}^{\infty}\frac{k\log^{2-\varepsilon} k}{q_k}
      \end{equation}
converges.
  \end{corollary} 
 \begin{proof}
By Corollary 5.14 (see the interval (90)), 
$q_k>2k^2\log^3 k$ for each $k\ge 2$. Therefore, 
$k\log^{2-\varepsilon}/q_k<1/(2k\log^{1+\varepsilon}k)$, and hence, 
     \begin{equation*}\begin{split}
\sum_{k=1}^{n}\frac{k\log^{2-\varepsilon} k}{q_k}
&<\frac{1}{2}\sum_{k=2}^{n}\frac{1}{k\log^{1+\varepsilon}k}
\sim \int_2^n\frac{d\,x}{x\log^{1+\varepsilon}x}\\
&=-\frac{1}{\varepsilon\log^{\varepsilon} x} \Big|_2^{n}=
\frac{1}{\varepsilon\log^{\varepsilon} 2}-
\frac{1}{\varepsilon\log^{\varepsilon} n}
\to \frac{1}{\varepsilon\log^{\varepsilon} 2}
\,\, {\rm as\,\,} n\to\infty.
      \end{split}\end{equation*}
Therefore, the series (97) converges.    
  \end{proof}

\section{Computational results}

By using {\tt Mathematica 8}, 
here we present our computational results concerning 
the the number  of   expression  ``prime sums'' $q_k$ (under Conjecture 
3.3) and related expression (the equality (56) of Theorem 4.4).
The  notion $\pi_n:=k$  
in the second column of Tables 1, 3 and 4   presents the 
number of primes in a set ${\mathcal S}_n:=\{S_1,S_2,\ldots ,S_n\}$,
where $n$ is a related value given in the first column of this table.
Hence, under notations of Section 1 and Conjecture 3.3,
     \begin{equation*}
k:=\pi_n:=\pi_{(S_k)}(S_n)=\# \{p:\, p\,\,{\rm is\,\,a\,\, prime\,\,
and}\,\, p=S_i \,\,{\rm for\,\,some\,\,} i \,\,{\rm with\,\,} 1\le i\le n\}.
   \end{equation*}
Accordingly, the value $k$ in the second column of Table 1 presents the 
number of primes in a set ${\mathcal S}_n$,
where $n$ is a related value given in the first column of this table. 
The appropriate rounded value of the greatest prime $q_k$ in 
${\mathcal S}_n$ is given in the third column 
(related exact values are given in Table 3), while in the next column it is 
written the values $n-m$, where 
$m$ are related indices such  that $q_k=S_m$. In the fifth column of 
Table 1 we present the corresponding  values of $M_k$ obtained as  solutions 
of the equation  (56) in Theorem  4.4. The refined upper bound $M_k^{(l)}$ 
and the upper bound $M_k^{(u)}$ of  $M_k$  given 
in Conjectures 5.13 and 5.7, respectively,
are given in the next two columns of Table 1. Notice taht the data from the 
 last two columns of this table are closely related to Conjecture 4.6 
and Conjecture 4.9, respectively.   
 
For example, a computation gives the following exact values:
  $q_{59129}=S_{849995}=22420773979207$,
$q_{62297}=S_{899999}=25235697805141$, $q_{2707378}=99262810294692679$
$q_{5212720}=S_{10^8}=411680592327546713$.
  \vspace{3mm}

\vfill\eject
 {\it Table} 1. Distribution of primes in the sequence $(S_n)$ in the 
range $1\le n\le 10^9+5\cdot 10^8$
  \vspace{2mm}
      \begin{center}
{\tiny{\rm
\begin{tabular}{ccccccccc}\hline
$n$ & $k:=\pi_n$ & $\approx q_k$ & $n-m$ & $M_k$ & $M_k^{(l)}$ 
& $M_k^{(u)}$ & $(k\log k)/m$ & $S_m\sqrt{k\log k}/$   \\
 &  & &  & & & & & $(2m^{5/2}\log m)$   \\\hline
100 & 23 & 107934 &1& 1.17894 & 1.22166  & 1.27421 &0.737366  &1.034203\\
1000 & 141 &15501706 &22& 1.18281 & 1.18140 & 1.20278 & 0.713472  & 0.995174\\
10000 &1098 & $2.12\cdot 10^{10}$ & 17 &1.14356 &1.16216 & 1.17714 
& 0.770046 &1.018392 \\
100000 & 8350& $2.64\cdot 10^{11}$ &10 & 1.15163 & 1.14863 & 1.16175& 
0.752577 &0.995122\\
155000 & 12379& $6.57\cdot 10^{11}$ & 3  &1.15259 & 1.14628 & 1.15912 &
0.752638 & 0.993165\\
185000& 14482 & $9.49\cdot 10^{11}$ & 6 &1.15241 & 1.14536 & 1.15821 &
0.750009  &0.990628\\
200000 & 15504 & $1.1\cdot 10^{11}$  & 18 & 1.15241 & 1.14495 & 1.15768
&0.750177 & 0.990392\\
220000 & 16954 & $1.36\cdot 10^{12}$ & 17 & 1.15090 & 1.14446 & 1.15714
& 0.752476 &0.991502\\
296000 & 22327 & $2.51\cdot 10^{12}$ & 3 & 1.15773 & 1.14302 & 1.15565
&0.741044 & 0.982660\\
300000 & 22595 & $2.59\cdot 10^{12}$ & 5 & 1.15774  & 1.14295 & 1.15556
&0.740997  & 0.982557\\
350000 & 26038  &$3.56\cdot 10^{12}$ & 4  & 1.15675  & 1.14216 & 1.15469 
 &0.742338 & 0.982816 \\
400000 & 29495  & $4.7\cdot 10^{12}$ & 49 & 1.15398 & 1.14147 &1.15399 & 
0.746555 & 0.985037\\
450000 & 32928  & $6.00\cdot 10^{12}$& 39  &1.15169 & 1.14087 &1.15336 
&0.750052 & 0.986842\\
500000 & 36302  & $7.47\cdot 10^{12}$& 14 & 1.15027 & 1.14035 &1.15274 
& 0.752199 &0.987819  \\
550000 & 39788  & $9.10\cdot 10^{12}$& 1 & 1.14730  &1.13984 & 1.15221&
0.756907 & 0.990517 \\
600000 & 43119  & $1.09\cdot 10^{13}$&  8  &1.14633   &1.13941 &1.15179 
&0.758358  &0.991099\\
650000 & 46488 & $1.28\cdot 10^{13}$& 13 & 1.14485   & 1.13902 &1.15136 
& 0.760670 & 0.992278\\
700000 & 49834 & $1.50\cdot 10^{13}$& 6  &1.14361   &1.13866 &1.15094&
0.762604 & 0.993234 \\
800000 & 56419 & $1.97\cdot 10^{13}$& 27  &1.14232   &1.13801 &1.15025
&0.764538 & 0.993953\\
850000 & 59602 & $2.24\cdot 10^{13}$& 5   & 1.14239   & 1.13772 & 1.14992 
&0.764332 & 0.993576\\
900000 & 62770 & $2.52\cdot 10^{13}$&  1  & 1.14245 &1.13746 &1.14964 
&0.764154 & 0.993230\\
950000 & 66064 & $2.82\cdot 10^{13}$& 45   & 1.14140  & 1.37190 & 1.14937
&0.765800 & 0.994084\\
$10^6$ & 69251 & $3.13\cdot 10^{13}$& 5   & 1.14093  & 1.13692 & 1.14910
&0.771847 & 0.950112\\
$2\cdot 10^6$ & 131841   & $1.31\cdot 10^{14}$& 23  
& 1.13400  & 1.13373 & 1.14563	 & 0.777169 & 0.998482\\
$3\cdot 10^6$ & 192655  & $3.03\cdot 10^{13}$& 14   
& 1.13116  & 1.13193 & 1.14366 & 0.781454 & 0.999694\\
$4\cdot 10^6$ & 252028  & $5.49\cdot 10^{14}$& 23   & 1.12965  
& 1.13069 & 1.14232 & 0.783641 & 1.0000075\\
$5\cdot 10^6$ & 310756  & $8.70\cdot 10^{14}$&  22  & 1.12809  
& 1.12973  & 1.14127  & 0.786405  & 1.0007040\\
$10^7$ & 594851  & $3.62\cdot 10^{15}$& 6   & 1.12473  & 1.12686 & 1.13814& 
0.790918 & 1.001331\\
$5\cdot 10^7$ & 2707378 & $9.92\cdot 10^{16}$& 10  & 1.11727  
& 1.12067 
& 1.131310& 0.802006 & 1.002899\\
$10^8$ & 5212720  & $4.11\cdot 10^{17}$& 13   & 1.11444  & 1.11819 & 1.12856 & 
0.806231 & 1.003355\\
$10^8+5\cdot 10^7$ & 7650550  & $9.45\cdot 10^{17}$& 13 &
1.11295  & 1.116783 & 1.126996 & 0.808423 & 1.003481\\
$2\cdot 10^8$ & 10047823 & $1.70\cdot 10^{18}$&  8    &
1.11189  & 1.11581 & 1.12591 & 0.809999 & 1.003599\\
$3\cdot 10^8$ & 14763858 & $3.91\cdot 10^{18}$& 4    &
1.11032  & 1.11446 & 1.12441 & 0.812391 & 1.003890\\
$4\cdot 10^8$ & 19404439  & $7.05\cdot 10^{18}$& 1    &
1.10922 & 1.11353 & 1.12336 & 0.814065 & 1.004097\\
$5\cdot 10^8$ & 23985388  & $1.11\cdot 10^{19}$& 13    &
1.10848 & 1.11281 & 1.12256 & 0.815165 & 1.004142\\
$7\cdot 10^8$ &   33031264  & $2.21\cdot 10^{19}$& 18    &
1.10730 & 1.11176 & 1.12138 & 0.816956 & 1.004305\\
$10^9$ & 46388006  & $4.60\cdot 10^{19}$& 2 
&1.10605 & 1.11066 & 1.12014 & 0.818867 & 1.004501\\
$10^9+5\cdot 10^8$ & 68259534  & $1.05\cdot 10^{20}$& $35$ 
&1.10473 & 1.10957 & 1.11877 & 0.820881 & 1.004650\\
  \end{tabular}}}
  \end{center}
\vspace{2mm}

Recall that $\pi(x)$ denotes the number of primes less or equal to $x$.
Then  Table 2 present the quotients  
$\frac{\pi_n}{\frac{n}{\log n}}$ and $\frac{\pi_n}{\pi(n)}$.
  
       \vspace{2mm}

{\it Table} 2. Distribution of primes in the sequence $(S_n)$ in the 
range $1\le n\le 10^9+5\cdot 10^8$

  \vspace{2mm}
      \begin{center}
{\tiny{\rm
\begin{tabular}{cccccc}\hline
 $n$ &  $\frac{\pi_n}{\frac{n}{\log n}}$ &  $\frac{\pi_n}{\pi(n)}$ & 
$n$ &  $\frac{\pi_n}{\frac{n}{\log n}}$ & $\frac{\pi_n}{\pi(n)}$ \\ \hline
  $10^2$  & 1.059190  & 0.920000   & $10^7$   & 0.958787   & 0.895079\\
  $10^3$   &  0.973993  &  1.011300  & $5\cdot 10^7$   & 0.959903  & 0.902118\\
  $10^4$    & 1.011300  & 0.893409  & $10^8$   & 0.960219 & 0.904758\\
  $10^5$ & 0.961329  & 0.870517   & $5\cdot 10^8$   & 0.960860  & 0.910059\\
  $10^6$  &  0.956738  &  0.882201   & $10^9$   & 0.961311  & 0.912296\\
  $5\cdot 10^6$  &  0.958679  &  0.891663  & $10^9+5\cdot 10^8$ & 0.961492
    & 0.913458 \\     
 \end{tabular}}}
  \end{center}

Notice that from Table 1 we see that 
$M_k^{(l)}$ is probably the upper bound of $M_k$
for $n>5\cdot 10^7$ (Conjecture 5.13)  which is better estimate 
 than $M_k^{(u)}$ (Conjecture 5.7). 

The values of first three columns of Table 3 are defined in the same 
way as these of Table 1 (with exact values of $q_k$), and the related values 
of ratios $q_k/(2k^2\log^3 k)$ are given in fourth  column of this table.
The asymptotic relation (87) of Corollary 5.10 shows that it can be 
of interest to analyze the sequence $(Q_k)_{k=2}^{\infty}$  
whose $k$th term is defined by the equality  
  \begin{equation}\label{(98)}
q_k=2k^2\log^3k+2Q_kk^2(\log^2 k)(\log\log k), \quad k=2,3,\ldots,
   \end{equation} 
or equivalently,
       \begin{equation}\label{(99)}
\frac{q_k}{2k^2\log^3k}=1+Q_k\cdot 
\frac{\log\log k}{\log k}\quad k=2,3,\ldots. 
   \end{equation}
We also consider two similar sequences $(Q_k')$ 
and $(Q_k'')$ which are closely  related to Corollary 4.7 
and Theorem 4.4, respectively (cf. Remark 4.11), and they are defined as 
      \begin{equation}
Q_k'=\frac{q_k-2k^2(\log^2 k)(\log k+\log\log k)}{2k^2(\log^2 k)\log\log k}, \quad k=2,3,\ldots,
   \end{equation}
and
      \begin{equation}
Q_k''=\frac{q_k-2(p_k)^2\log k}{2k^2(\log^2 k)\log\log k}, \quad k=2,3,\ldots.
    \end{equation}
Some values of these  sequences are given in the last three columns of Table 3.
     \vspace{2mm}

 {\it Table} 3. Some ``prime sums'' $q_k$'s in the sequence $(S_n)$ 
with $n\le 10^9+5\cdot 10^8$ and related values 
$q_k/(2k^2\log^3 k)$, $q_k/(2m^2\log m)$, $Q_k$, $Q_k'$ and $Q_k''$

\vspace{2mm}
 \begin{center}
  {\tiny
\begin{tabular}{cccccccc}\hline
$n$ & $k:=\pi_n$ & $q_k$ & $\frac{q_k}{2k^2\log^3 k}$ & 
$\frac{q_k}{2m^2\log m}$ 
& $Q_k$ &  $Q_k'$ & $Q_k''$ \\\hline
10 &  5  & 281  & 1.34807 & 1.47352  & 2.35436 & 0.17772  & -1.76013 \\
100 & 23   & 107934  &  3.30944 & 1.19829 & 6.33647 & 5.33647 & 5.44583 \\
1000 & 141    & 15501706  & 3.21679  & 1.17689  & 6.86016 & 5.86016 & 5.77437 \\
100000 & 8350 & 264074170741 & 2.58273 & 1.14710 & 6.49405 & 5.49405 & 5.28293 \\
200000 & 15504 & 1116374522657 & 2.56968 & 1.14347 & 6.68238 & 5.68238  & 5.44027\\
300000 & 22595 & 2591079720139 & 2.61956 & 1.14145 & 7.03697 & 6.03697  & 5.79201\\
400000 & 29495  & 4704619172003 & 2.56741 &1.14004  & 6.91365 & 5.91365 
& 5.66455\\
500000 & 36302  &  7472533368077 & 2.51877 & 1.15867  & 6.77736 & 5.77736 
&5.51158\\
600000 & 43119 &  10901967324637   & 2.46956 & 1.13810  & 6.62021 &5.62021 
&5.35079\\
700000 & 49834 &  15001269948023 &2.43548 & 1.13737 & 6.51781 & 5.51781 
&5.24772\\
800000 & 56419 & 19776121232971 & 2.41801 & 1.13675 & 6.48149 & 5.48149 &
5.2025\\
900000 & 62770 & 25235697805141 & 2.41648 & 1.13621 & 6.51155 & 5.51155 &
5.22977\\
$10^6$ & 69251  &   31380813002879 & 2.40459 & 1.13572 & 6.30126 & 
5.30126 & 5.02105\\  
$4\cdot 10^6$ & 252028   &  549524547523421 & 2.24844 & 1.12966   
& 6.15989 & 5.15989 & 4.84045\\  
$5\cdot 10^6$ & 310756   &   870522520170287 & 2.22830 & 1.12873   
& 6.12200 & 5.12200 & 4.79728\\  
$10^7$ &  594851  &  3629567501866919 & 2.181921 & 1.12593 & 6.07346 &
5.07346 & 4.7374\\  
$5\cdot 10^7$ & 2707378 & 99262810294692679 & 2.083831 & 1.11987  
& 5.95575 & 4.95575 & 4.58936 \\  
$7\cdot 10^7$ & 3720648 & 198036667738658321  &  2.065440  & 1.11868 
& 5.93361 & 4.93361 & 4.56144\\  
$8\cdot 10^7$ & 4220531 & 260463664887226043& 2.059235& 1.11821  
& 5.93009  & 4.93009 & 4.55643\\  
$10^8$ & 5212720 & 411680592327546713  & 2.047463  & 1.11744 & 5.91551  
& 4.91551 & 4.53769\\  
$2\cdot 10^8$ & 10047823  &1705122556732581169 & 2.014906 & 1.11511 & 
5.88554 & 4.88554 & 4.49873\\  
$3\cdot 10^8$ & 14763858  & 3913274710820657161& 1.995499 & 1.11379 & 
5.86106 & 4.86106 & 4.46806\\  
$4\cdot 10^8$ & 19404439 & 7053651472078078383& 1.982108 & 1.11287 & 
5.84373 & 4.84373 & 4.44675\\ 
$5\cdot 10^8$ & 23985388  &  11138479445180255153& 1.972857 & 1.11217 & 
5.83583 & 4.83583 & 4.43625\\ 
$7\cdot 10^8$ & 33031264  & 22177401605086098829& 1.958466 & 1.11113 & 
5.81944 & 4.81944 & 4.41565\\ 
$10^9$ & 46388006  & 46007864234123508181& 1.943427 & 1.11005 & 
5.80097 & 4.80097 & 4.39275\\ 
$10^9+5\cdot 10^8$ &  68259534  & 105428905479616558423 & 1.927428 
& 1.10885 & 5.78377 & 4.78377 & 4.37116\\ 
   \end{tabular}}
  \end{center}

It is easy to prove the folowing result.
  \begin{proposition}
Let $(Q_k)$, $(Q_k')$ and $(Q_k'')$ be the sequnces 
defined by $(99)$, $(100)$ and $(101)$, respectively. Then 
       \begin{equation*}
   \lim_{k\to\infty}(Q_k-Q_k')=\lim_{k\to\infty}(Q_k'-Q_k'')=0.
        \end{equation*}
  \end{proposition}
We also propose the following conjecture.
 \begin{conjecture} 
The all sequences  $(Q_k)$, $(Q_k')$ and $(Q_k'')$ converge 
to $1$. 
 \end{conjecture}
  \begin{remark}
In view of the above conjecture, it can be of interest to consider the 
sequence  $(Q_k''')$ defined as
     \begin{equation*}
Q_k'''=\frac{q_k-2k^2(\log^2 k)(\log k+\log\log k)}{2k^2(\log^2 k)\log\log\log k}, 
\quad k=2,3,\ldots,
   \end{equation*}
The values of $Q_{10^s}'''$ for $s=3,5,6,7,8,9$    are  
equals to $19.961$, $15.329$  $14.524$ $13.809$  $13.621$, $13.069$,
respectively.  
  \end{remark}

  \begin{remark} 
The values $V_k:=q_k/(2m^2\log m)=S_m/(2m^2\log m)$
presented in the fourth column of Table 3 are in fact terms 
of the sequence $t_n:=S_n/(2n^2\log n)$ with $n=2,3,\ldots,$ which is 
decreasing under Conjecture 5.3 on the range $\{ n\in \Bbb N:\, 
n\ge 1100\}$. Accordingly, under 
Conjectures 4.6 and 5.3 and the fact that 
$q_{151}=S_{1100}=19949537$, we immedately get 
   $$
V_k=t_m \le  t_{\lfloor k\log k\rfloor +1}<
t_{\lfloor k\log k\rfloor}:=M_k^{(l)}\quad {\rm for\,\, all\,\,}
k\ge 151,
   $$
that is, 
  \begin{equation}\label{(100)}
V_k<M_k^{(l)}\quad {\rm for\,\, all\,\,}k\ge 151,
 \end{equation}
where $M_k^{(l)}$ are  approximative values for $M_k$ given by 
(89) and  presented in Table 1. 

Moreover, the comparison of values of $M_k$ with those of  $V_k$
from Tables 1 and 2, respectively, 
leads to the following conjecture.
 \end{remark}
 
 \begin{conjecture}   Let $(M_k)_{k=2}^{\infty}$ be the 
sequence defined by {\rm(56)} of Theorem {\rm 4.4},
and let $m(k)=m$ be defined as $S_m=q_k$. Then  
 \begin{equation}\label{(101)}
M_k<\frac{q_k}{2m^2\log m}\quad for\,\, all\,\,k\ge 4\times 10^6.
 \end{equation}
  \end{conjecture}

Consequently, we obtain
the following  ``weak version'' of Conjecture 6.5.

 \begin{corollary}   Let $(M_k)_{k=2}^{\infty}$ be the 
sequence defined by {\rm(56)} of Theorem {\rm 4.4},
and let $m(k)=m$ be defined as $S_m=q_k$. Then  
under Conjectures {\rm 4.6}, {\rm 5.3} and {\rm 6.5} we have
 \begin{equation*}
M_k<t_{\lfloor k\log k \rfloor+1}:=
\frac{S_{\lfloor k\log k \rfloor+1}}{2(\lfloor k\log k \rfloor+1)^2
\log (\lfloor k\log k \rfloor+1)}
\quad for\,\, all\,\,k\ge 4\times 10^6. 
 \end{equation*}
  \end{corollary}
 \begin{proof}
Combining Conjectures {\rm 4.6}, {\rm 5.3} and {\rm 6.5}, we find that 
for  all $k\ge 4\times 10^6$ with $q_k=S_m$ 
   \begin{equation*}
M_k<\frac{q_k}{2m^2\log m}=\frac{S_m}{2m^2\log m}=
t_m\le t_{\lfloor k\log k \rfloor +1}=
\frac{S_{\lfloor k\log k \rfloor+1}}{2(\lfloor k\log k \rfloor+1)^2
\log (\lfloor k\log k \rfloor+1)},
  \end{equation*}
as desired.
 \end{proof}

  \begin{remark} 
The ratios  $L_k:=q_k/(2k^2\log^2k(\log k+\log\log k))$
are closely related to Corollary 4.7. Of course, the values $L_k$ are 
small than the related values $q_k/(2k^2\log^3 k)$ presented 
in the fourth column of Table 3. For example, 
$L_k$ is equal to $1.762999, 1.696920$ for $k=2707378,19404439$, 
respectively. However, the sequence  $(L_k)$ slowly tends to 1 as $k$ grows.
This is directly connected with the fact that the sequence $(k\log k/m(k))$
converges very slowly to 1 as $k$ grows (see the eighth column of Table 1).
    \end{remark} 
  \begin{remark} 
A good approximation from Remark 4.11 arising from the last column 
of Table 1 can be written as 
   \begin{equation}\label{(104)}
\sqrt{k\log k}\approx \frac{2m^2\sqrt{m}\log m}{S_m},
   \end{equation}
where $q_k=S_m$. The approximation (104) allows us 
 for given $n$ to  determine  the index $k=k(n)$ such that 
the prime sum $q_k$ is ``very close'' to $S_n$; especially, for 
each $n\ge 4\times 10^6$ (i.e., for $k\ge 252028$), assuming that  Conjecture 
4.9 is true, then $q_{k(n)}<S_n$. Accordingly, for given $n$ we assume 
that $k_0(n)=\lfloor x_0\rfloor$, where $x_0=x_0(n)$ is a root 
of the equation 
    \begin{equation}\label{(105)}
\sqrt{x\log x}=\frac{2n^2\sqrt{n}\log n}{S_n}.
   \end{equation}
For some values $n$ from Table 1 Table 4 presents
  the exact largest values $k(n)$  
such that $q_{k(n)}\le S_n$ (these values are in fact, given in
the second column of Table 1) and related differences $k(n)-k_0(n)$.\\
 
\end{remark}

 {\it Table} 4. 
The values $k=k_0(n)$ and $k(n)-k_0(n)=k-k_0$ for some values  
$n\le 10^9$

 \vspace{2mm}
   \begin{center}
{\tiny
\begin{tabular}{ccccccccc}\hline
 $n$ & $k:=\pi_n$ & $k-k_0$ &  $n$ & $k$ & $k-k_0$ &  $n$ & $k$ & $k-k_0$\\\hline
$10$ & 5  & 3  & $10^2$  & 23  & 1 & $10^3$ & 141 & -3\\
$10^4$ & 1098    & -33   & $10^5$  & 8350   & -59  & $2\cdot 10^5$ 
& 15504 & -315\\
$3\cdot 10^5$ & 22595   & -338    & $4\cdot 10^5$  & 29495  & -371  
& $5\cdot 10^5$ &  36302  & -812\\
$6\cdot 10^5$ & 43119  & -263  & $7\cdot 10^5$  & 49834  &-177  & 
$8\cdot 10^5$ & 56419   & -627\\
$9\cdot 10^5$ & 62770   & -308   & $10^6$  & 69251  & -283  & 
$2\cdot 10^6$ & 131841  & -368\\
$3\cdot 10^6$ & 192655  &  -110  & $4\cdot 10^6$  & 252028  & {\bf 5}   & 
$5\cdot 10^6$ & 310756  & 404\\
$10^7$ & 594851   & 1469  & $2\cdot 10^7$  & 1141478  & 4638   & $3\cdot 10^7$ &
1671839  & 7462 \\
$4\cdot 10^7$ & 2193083   & 10997   & $5\cdot 10^7$  & 2707378   & 14644  & 
$6\cdot 10^7$ & 3216515   & 18621 \\
$7\cdot 10^7$ & 3720648  & 22061  & $8\cdot 10^7$  & 4220531  & 25021  & 
$9\cdot 10^7$ & 4717545  & 28357  \\
$10^8$ & 5212720  & 32696   & $10^8+ 10^7$  & 5703356  & 35030    
& $10^8+2\cdot 10^7$ & 6191655  & 37303\\
$10^8+3\cdot 10^7$ & 6679364  & 41059  & $10^8+ 4\cdot 10^7$ & 7165567 & 45196  & 
$10^8+5\cdot 10^7$ & 7650550  & 49854\\
$10^8+6\cdot 10^7$ & 8132623  & 53221   & $2\cdot 10^8$  & 10047823  & 67743
 & $3\cdot 10^8$ & 14763858  & 107704\\
$4\cdot 10^8$ & 19404439   & 149159  & $5\cdot 10^8$  & 23985388  &   
186542 & $7\cdot 10^8$ & 33031264  & 267196 \\
$8\cdot 10^8$ & 37508452   & 309262  & $9\cdot 10^8$ & 41960355 & 351779  
& $10^9$ & 46388006    & 392660\\
\end{tabular}}
   \end{center}
\vspace{2mm}

In view of the above considerations and computational results given 
in Table 4, we propose the following conjecture.

 \begin{conjecture}

Let $n\ge 4\times 10^6$ be a positive integer, and let 
$x_0(n)$ be a real root of the equation 
    \begin{equation}\label{(106)}
\sqrt{x\log x}=\frac{2n^2\sqrt{n}\log n}{S_n}.
   \end{equation}
Then the set $\{S_1,S_2,\ldots ,S_n\}$ contains at least 
$\lfloor x_0(n)\rfloor$ primes.
  \end{conjecture}
The inequality on right hand side of (37) of Proposition 3.12 
immediately gives the following weak version of Conjecture 6.9.

 \begin{conjecture}
Let $n\ge 4\times 10^6$  be a positive integer, and let 
$y_0(n)$ be a real root of the equation 
    \begin{equation}\label{(107)}
\left(1+\frac{\log 2+\log\log (2n)}{\log n}\right)\sqrt{y\log y}=\sqrt{n}.
   \end{equation}
Then the set $\{S_1,S_2,\ldots ,S_n\}$ contains at least 
$\lfloor y_0(n)\rfloor$ primes.
  \end{conjecture}
It can be also of interest to compare the values 
$k_0(n)$ and $k_1(n):=\lfloor y_0(n)\rfloor$ with the values $k_2(n):=\lfloor z_0(n)\rfloor$,
where $z_0(n)$ is a real root of the equation 
    \begin{equation}\label{(108)}
x\log x=n.
   \end{equation}
 \begin{corollary}
Let $n\ge 4\times 10^6$  be a positive integer. Then under Conjecture 
{\rm 6.10} and its notations, the sequence $(S_i)_{i=1}^{\infty}$ 
contains at least $k_0(n):=\lfloor y_0(n)\rfloor$ primes 
which are less than $2n^2(\log n+\log\log(2n)+\log 2)$. In other words,
      \begin{equation}\label{(1109)}
q_{k_0(n)}<2n^2(\log n+\log\log(2n)+\log 2).
    \end{equation}
 \end{corollary}
 \begin{proof}
The assertion immediately follows from Conjecture 6.10 and 
the right hand side of the inequalities  (37) from Proposition 3.12.
\end{proof}

The values $k_0(n)$ (derived from Table 4 as 
the differences $k_0(n)=k(n)-(k(n)-k_0(n))$), 
$k_1(n)$ and $k_2(n)$ concerning the values of $n$ from Table 4, 
are presented in Table 5.

\vspace{2mm}

{\it Table} 5. The values $k_i(n)$, $i=0,1,2$, the ratios 
$\delta_j(n):=(k(n)-k_j(n))/k(n)$  with $j=1,2$,
 the ratios $\eta(n):=k_0(n)/\sqrt{k_1(n)k_2(n)}$ and 
$\xi(n):=k(n)/\sqrt{k_1(n)k_2(n)}$ 

 \vspace{2mm}
    \begin{center}
{\tiny
\begin{tabular}{ccccccccc}\hline
 $n$  & $k_0(n)$ & $\delta_0(n)$  &  $k_1(n)$  &
$\delta_1(n)$ & $k_2(n)$  &  $\delta_2(n)$ & $\eta(n)$ & $\xi(n)$ \\\hline
 $10$   & 2   & 0.60000  & 2 & 0.60000  & 5 & 1.00000 & 0.63246 & 1.58114 \\
 $10^2$ &  22  &0.04348  & 15  & 0.34783  & 29  & -0.26087& 1.05482 &1.10277 \\
 $10^3$ &  144  & -0.00213  & 109  & 0.22695  & 190 & -0.347552& 1.00063 & 0.97978 \\
 $10^4$ & 1131   &  -0.03005 &846  & 0.22951 & 1382  & -0.25865& 1.04598 & 1.01546 \\
 $10^5$ & 8409  & -0.0071 & 6928  & 0.17030  & 10770  &-0.28982 & 0.97345 
&0.96666 \\
 $5\cdot 10^5$ & 37114   & -0.02237 & 30816 & 0.15112 & 46521   &-0.28150& 
0.98022  & 0.95878\\
  $10^6$ & 69534  & -0.00409  & 58857  & 0.15009  & 87845  & -0.26850& 0.96703 &
0.96309 \\
  $4\cdot 10^6$ & {\bf 252023}  & 0.00002  & 216103  & 0.14254  & 315878  &
-0.25334 & 0.96461 & 0.96463\\
  $5\cdot 10^6$ & 310352  & 0.00130  & 266622   & 0.14202  & 388499 &-0.25017 
& 0.96430 &0.96555\\
 $10^7$ & 593382  & 0.00247 & 512630  & 0.13822  & 739955   & -0.24393 & 
0.96345 &0.96584\\
 $5\cdot 10^7$ &  2692734  & 0.00541 & 2353142   & 0.13084  & 3329279 
& -0.22971 & 0.962043 &0.96728\\
 $10^8$ &  5180024  & 0.0063  & 4546674  & 0.12778    & 6382029 & -0.22432& 
0.96162 &0.96769\\
 $5\cdot 10^8$ & 23798846   & 0.00778   & 21080800  & 0.12110  & 29093410 
&-0.21296 & 0.96098 &0.96851\\
 $10^9$ & 45995346 & 0.00846  & 40886757 & 0.11859 & 56048389 &-0.20825& 
0.96082 &0.96902
\end{tabular}}
    \end{center}
\vspace{2mm}

The last column of Table 5 suggests that 
$\xi(n)<1$ for all $n\ge 10^5$, which is obviously 
equivalent with the following conjecture.

 \begin{conjecture}

Let $n\ge  10^5$ be a positive integer, and let 
$k_1(n)=\lfloor y_0(n)\rfloor$ and $k_2(n)=\lfloor z_0(n)\rfloor$,
where $y_0(n)$ and $z_0(n)$  are real roots of the equations
{\rm (107)} and {\rm (108)}, respectively.
 Then the set $\{S_1,S_2,\ldots ,S_n\}$ contains less than
$\lfloor \sqrt{k_1(n)\cdot k_2(n)}\rfloor$ primes.
  \end{conjecture}

As an immediate consequence, we obtain the following statement.

 \begin{corollary}
Let $n\ge  10^5$  be a positive integer. Then under Conjecture 
{\rm 6.12} and its notations the sequence $(S_i)_{i=1}^{\infty}$ 
contains at most $\lfloor \sqrt{k_1(n)\cdot k_2(n)}\rfloor$ primes 
which are less than $2n^2\log n$. In other words,
   \begin{equation}\label{(110)}
q_{\lfloor \sqrt{k_1(n)\cdot k_2(n)}\rfloor}> 2n^2\log n.
    \end{equation}
 \end{corollary}
 \begin{proof}
The assertion immediately follows from Conjecture 6.12 and 
left  hand side of the inequalities of (37) from Proposition 3.12.
\end{proof}
 
 \begin{remark}
Conjecture 6.6 may be considered as the ``prime sums analogue''
of the well known fact that $p_k\ge k\log k$ for all $k\ge 3$,
where $p_k$ is the $k$th prime (see e.g., \cite[p. 69]{rs1})
 \end{remark}

 \begin{remark} The approximation (105) can be written as 
 \begin{equation}\label{(111)}
\sqrt{\frac{m}{k\log k}}\approx \frac{S_m}{2m^2\log m},
   \end{equation}
which in view of Conjecture 5.3 asserts  that the sequence 
$(m/(k\log k))_{k=k_1}^{\infty}$ is decreasing 
for a fixed large integer $k_1$.  

On the other hand, if we write the estimate (111) in the form 
  \begin{equation}\label{(112)}
\sqrt{\frac{m\log^2 m}{k\log k}}\approx \frac{S_m}{2m^2},
   \end{equation} 
then Proposition 5.1 suggests that 
the sequence $(m\log^2 m/(k\log k))_{k=k_2}^{\infty}$ is increasing 
for some fixed large integer $k_2$.  
  \end{remark}

  \section{The stronger form of 
Conjecture 3.3}
Around 1800, young C. F. Gauss conjectured that for large $n$ the 
the number of primes not exceeding $n$ is nearly
  \begin{equation*}
{\rm li}(n):=\int_{2}^{n}\frac{dt}{\log t}.
    \end{equation*}
Heuristic and computational arguments give the impression that 
Restricted Prime Number Theorem (RPNT) for the sequence $(S_n)$
(i.e., Conjecture 3.3) probably holds in its stronger form which 
in fact presents the well known form of  Prime Number Theorem (PNT)
 for primes 
(see e.g., \cite[Chapter 12]{iv1}). Accordingly, we propose
the following conjecture.
   \begin{conjecture}
 Let $\pi_{(S_k)}(S_n)=\pi_n$ be the number of primes $p$ in the sequence
$(S_k)$ such that $p=S_i$ for some $i$ with $1\le i\le n$. Then
   \begin{equation}\label{(113)}
\pi_n={\rm li (n)} +R(n),
    \end{equation}
where
  \begin{equation}\label{(114)}
{\rm li}(n):=\int_{2}^{n}\frac{dt}{\log t}=\frac{n}{\log n}+
O\left(\frac{1}{\log^2 n}\right)\quad as \,\, n\to\infty. 
  \end{equation}
is the logarithmic integral and 
  \begin{equation}\label{(115)}
 R(n)\ll ne^{(-C\delta(n))}\quad with \quad \delta(n):=
(\log n)^{3/5}(\log\log n)^{-1/5}.
 \end{equation}
  \end{conjecture}
Assuming the above conjecture, and following related 
``PNT  result'' of A. Ivi\'{c}  and J.-M. De Koninck \cite[Theorem 9.1]{ik}
(see also \cite[Theorem]{iv2}), it can be proved the following   result.
     \begin{corollary}
Under  the  truth and notations of Conjecture {\rm 7.1}, we have
  \begin{equation}\label{(116)}
 \sum_{i=1}^n\frac{1}{\pi_i}=\frac{1}{2}\log^2 n+O(\log n)\quad
as \,\, n\to\infty.
  \end{equation}
  \end{corollary}
Similarly, using Conjecture 7.1 it can be proved  the
following result. 
  \begin{corollary} 
Let $q_k$ be the $k$th prime in $(S_n)$. Then under Conjecture $7.1$, 
    \begin{equation}\label{(114)}
\sum_{k=1}^{n}\frac{k\log^2 k}{q_k}=\log\log n+
o\left(\frac{1}{\log n}\right).
  \end{equation}
 \end{corollary} 

Finally,  we propose the following  conjecture.
\begin{conjecture}[Chebyshev inequalities for $(S_n)$]
There exist positive 
constants $c_1$, $c_2$ and 
a positive positive integer $n_0$ such that
     \begin{equation}\label{(115)}
\frac{c_1n}{\log n}\le\pi_{(S_n)}(S_n)\le \frac{c_2n}{\log n}\quad
 for\,\, all\,\, n>n_0.
    \end{equation}
\end{conjecture}

\begin{remark} We also believe that for the sequence 
$(S_n)$ are valid the analogues of some other classical results 
and conjectures closely related to the Prime Number Theorem. 
\end{remark}

\begin{remark} 
Numerous computational results involving 
 sums of the first $n$ primes (the Sloane's sequence 
A007504 sequence here denoted as $S_n'$) 
and certain their curious arithmetical properties  are 
presented in Sloane's sequences A051838 (numbers $n$ such that 
sum of first $n$ primes divides product of first $n$ primes), A116536, 
A067110, A067111, A045345, A114216 (sum of first $n$ primes divided by maximal power 
of 2), A024011 (numbers $n$ such that $n$th prime divides 
sum of first $n$ primes), A036439 ($a(n)=2+$ the sum of the first 
$n-1$ primes), A014284 (partial sums of primes, if 1 is regarded 
as a prime; $1,3,611,18,29,\ldots$), A134125 (integral 
quotients of partial sums of primes divided 
by the number of summations; $5,5,7,11,16,107,\ldots$), 
A134126 (indices $k$ such that the $(k+1)$th partial sum of primes 
divided by $k$ is integer; $1,2,4,7,10,50,130,\ldots$), 
A134127 (largest prime in the partial sums of primes in A134125
which have integer averages), A134129 (prime partial sums 
$A007504(k+1)$ such that $A007504(k+1)/k$ is integer; $5,10,28,77,160,
\ldots$), A077023, A033997, A071089, A083186 
(sum of first $n$ primes whose indices are primes), A166448 (sum of first
$n$ primes minus next prime), A196527, A065595 (square of 
first $n$ primes minus sum of squares of first $n$ primes), A165906
(sum of first $n$ primes divided by the $n$th prime), 
A061568 (number of primes $\le$ sum of first $n$ primes), 
A066039 (largest prime less than or equal to the sum of first $n$ primes),
A077022, A110997, A112997 (sum of first $n$ primes minus sum of their
indices), A156778 (sum of first $n$ primes multiplied by $n/2$),
A167214 (sum of first $n$ primes multiplied by $n$), A038346 (sum of first
$n$ primes $\equiv 1(\bmod{\,4})$, A038347 (sum of first 
$n$ primes $\equiv 3(\bmod{\,4})$, A054972 (product of sum of first
$n$ primes and product of first $n$ primes), A072476, 
A076570 (greatest prime divisor of sum of first $n$ primes), 
A076873 (smallest prime not less than sum of first $n$ primes), 
A077354 (sum of second  string of $n$ primes-sum of first
$n$ primes, or $2n$th partial sum of primes; this is in fact our sequence 
$(S_n)$),  A110996, A123119 (number of digits in sum of first
$n$ primes), A189072 (semiprimes in the sum of first $n$ primes),
A196528, A022094 (sum of first $p_n$ primes, where $p_n$ is the $n$th prime),
A024447, A121756, A143121 (triangle read by rows,
$T(n,k)=\sum_{j=k}^np_j$, $1\le k\le n$), A117842 (partial sum of smallest
prime $\ge n$), A118219, A131740 (sum of $n$ successive primes after
$n$th prime), A143215 (the sequence whose $n$th term is
$p_n\cdot S_n'=p_n\cdot \sum_{i=1}^np_i$), A161436 (sum of all primes from
$n$th prime to $(2n-1)$th prime), A161490,  A013918 (numbers $n$ such that
$n$ is prime and is equal to the sum of the first $k$ primes for
some $k$; $2,5,17,41,197,281,7699,8893,\ldots$) etc.
\end{remark}

\end{document}